\documentclass[12pt,a4paper]{article} 
\usepackage{amsmath}
\usepackage{amsfonts}
\usepackage{amssymb}
\usepackage{cite}
\usepackage{graphics}
\usepackage{amssymb,amscd}
\usepackage[utf8]{inputenc}
\usepackage[T2A]{fontenc}
\usepackage{tikz-cd}
\usepackage[russian,english]{babel}

\usepackage{amscd}
\usepackage{yfonts}
\usepackage[normalem]{ulem}
\usepackage{esint}
\usepackage{amsmath}
\usepackage{amsfonts}
\usepackage{amssymb}
\usepackage{hyperref}
\usepackage{verbatim}

\usepackage{graphics}
\usepackage{amsfonts,amssymb,amsmath,mathrsfs,amsthm,amscd}
\usepackage{tikz} 
\usepackage{tikz-cd} 
\usetikzlibrary{decorations.pathreplacing}

\theoremstyle{plain}

\newtheorem{theo}{Theorem}[section]
\newtheorem{prop}[theo]{Proposition}

\newtheorem{lemm}[theo]{Lemma}

\theoremstyle{definition}
\newtheorem{defi}{Definition}

\theoremstyle{remark}
\newtheorem*{rema}{Remark}

\numberwithin{equation}{section}

\DeclareMathOperator{\tr}{tr}
\DeclareMathOperator{\re}{Re}
\DeclareMathOperator{\Res}{Res}
\DeclareMathOperator{\Spec}{spec}
\DeclareMathOperator{\Scal}{Scal}
\DeclareMathOperator{\Ric}{Ric}
\DeclareMathOperator{\R}{R}
\DeclareMathOperator{\Vol}{vol}
\DeclareMathOperator{\dvol}{dvol}

\title{Heat trace expansion on manifolds with conic singularities.}
\author{Asilya Suleymanova}
\date{October 2017}

\begin{document}
\maketitle

\begin{abstract}
We derive a detailed asymptotic expansion of the heat trace for the Laplace-Beltrami operator on functions on manifolds with conic singularities, using the Singular Asymptotics Lemma of Jochen Br\"uning and Robert T. Seeley \cite{BS}. In the subsequent paper we investigate how the terms in the expansion reflect the geometry of the manifold.
\end{abstract}

\tableofcontents

\begin{section}{Introduction}

Consider a Riemannian manifold, $(M,g)$, of dimension $m$. The Laplace-Beltrami operator is, by definition, the Hodge Laplacian restricted to smooth functions on $(M,g)$. The space of smooth functions can be completed to the Hilbert space of square integrable functions. The Laplace-Beltrami operator, $\Delta$, is a symmetric non-negative operator and it always has a self-adjoint extension, the Friedrichs extension. We are interested in those Riemannian manifolds where the Friedrichs extension of the Laplace-Beltrami operator has discrete spectrum, $\Spec\Delta$, see Section~\ref{geometric setup}.

Spectral geometry studies the relationship between the geometry of $(M,g)$ and $\Spec\Delta$. One of the main tools of spectral geometry is the {\itshape heat trace} 
\begin{equation}\label{trace}
\tr e^{-t\Delta}=\sum_{\lambda\in\Spec\Delta}e^{-t\lambda}.
\end{equation}

For compact Riemannian manifolds $(M,g)$, the problem of finding geometric information from the eigenvalues of the Laplace-Beltrami operator and the Hodge Laplacian has been extensively studied, see e.g. \cite{G2} and the references given there. On closed $(M,g)$ there is an asymptotic expansion
\begin{equation}\label{smooth expansion}
\tr e^{-t\Delta}\sim_{t\to+0}(4\pi t)^{-\frac{m}{2}}\sum_{j=0}^{\infty} a_jt^j,
\end{equation}
where $a_j\in\mathbb{R}$.
In principle, every term in (\ref{smooth expansion}) can be written as an integral over the manifold of a local quantity. Namely, 
\begin{align}\label{smooth terms}
a_j=\int_M u_j\dvol_M,
\end{align}
where $u_j$ is a polynomial in the curvature tensor and its covariant derivatives, see Section~\ref{section local expansion}. In particular, $u_0=1$ and $u_1=\frac16\Scal$, where $\Scal$ is the scalar curvature of $(M,g)$. The bigger $j$, the more complicated the calculation of $u_j$. Sometimes we write $u_j(p)$ to indicate that it is a local quantity, i.e.~it depends on a point $p\in M$.

There are many examples of manifolds that are {\itshape isospectral}, i.e.~have the same spectrum of $\Delta$, but are not isometric, see the survey \cite{GPS}. However, it remains very interesting to study to what extent the geometry of $(M,g)$ can be determined from $\Spec\Delta$.

In this article we study the heat trace expansion of the Friedrichs extension of the Laplace-Beltrami operator on a non-complete smooth Riemannian manifold $(M,g)$ that possesses a conic singularity. By this we mean that there is an open subset $U$ such that $M\setminus U$ is a smooth compact manifold with boundary $N$. Furthermore, $U$ is isometric to $(0,\varepsilon)\times N$ with $\varepsilon>0$, where the {\itshape cross-section} $(N,g_N)$ is a closed smooth manifold, and the metric on $(0,\varepsilon)\times N$ is
\begin{align}\label{conic metric}
g_{\text{conic}}=dr^2+r^2g_{N}, \;\;\; r\in(0,\varepsilon).
\end{align}

\begin{figure}[!ht] 
\centering
\begin{tikzpicture}[scale=1]
    \draw (0,0) -- (3,1);
    \draw (0,0) -- (3,-1);
    \draw (0,0) circle [radius=0.07];
    \draw[dashed] (2.45,0.8) arc (90:270:0.4cm and 0.8cm);
    \draw (2.4,-0.8) arc (-90:90:0.4cm and 0.8cm);
    \draw (2.35,-1.1) node {$N$};
    \draw[decorate,decoration={brace,raise=3pt,amplitude=5pt}, thick] (0,2)-- (3,2);
    \draw (1.5,2.6) node {$U$};

    \draw (4,0) circle [x radius=0.4cm, y radius=1.2cm];
    \draw (4,1.2) .. controls (12,4) and (12,-4) .. (4,-1.2);
    \draw (3.9,-1.5) node {$N$};
    \draw[decorate,decoration={brace,raise=3pt,amplitude=5pt}, thick] (3.8,2)-- (10,2);
    \draw (6.9,2.6) node {$M \setminus U$};

    \draw (5.3,0) to[bend left] (6.3,0);
    \draw (5.2,.05) to[bend right] (6.4,.05);

    \draw (7.3,0.8) to[bend left] (8.3,0.8);
    \draw (7.2,0.85) to[bend right] (8.4,0.85);

    \draw (7.3,-0.8) to[bend left] (8.3,-0.8);
    \draw (7.2,-0.75) to[bend right] (8.4,-0.75);

    \draw[->] (-0.5,-2.3) -- (3.8,-2.3);
    \draw (0,-2.38) -- (0,-2.22); 
    \draw (2.45,-2.38) -- (2.45,-2.22);
    \draw (0,-2.6) node {$0$}; 
    \draw (2.45,-2.6) node {$\epsilon$};
    \draw (3.9,-2.6) node {$r$}; 

\end{tikzpicture}
\caption{Manifold with a conic singularity.}
\end{figure}
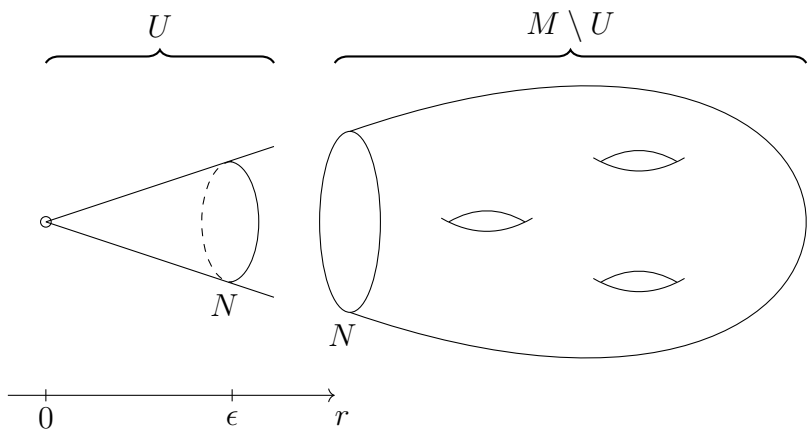

The existence of the heat trace expansion of the Friedrichs extension of the Hodge Laplacian on the differential forms on manifolds with conic singularities was proven by Jeff Cheeger \cite[Section~5]{Ch}. Jochen Br\"uning and Robert Seeley in \cite{BS} and \cite{BS2} developed a general method for showing the existence of the heat trace expansion of second order elliptic differential operators. A fundamental feature of the expansion on manifolds with conic singularities is that a logarithmic term can appear (see also \cite[(4.6)]{BKD}), while only power terms can appear in (\ref{smooth expansion}) on a smooth closed manifold. It was not fully understood how a singularity contributes to the coefficients in the expansion. In this article we used the local heat kernel expansion, see Section~\ref{section local expansion}, and then the Singular Asymptotics Lemma from \cite[p.\,372]{BS2}, see Section~\ref{regularized integrals}, to compute the terms in the heat trace expansion on a manifold with a conic singularity.

The negative power terms in the expansion do not have any contribution from the singularity and are computed for a bounded cone with different boundary conditions by Michael Bordag, Klaus Kirsten and Stuart Dowker in \cite[(4.7)--(4.8)]{BKD}. The first power term in the expansion that is affected by the singularity is the constant term. The expression of the constant term in \cite[Theorem~4.4]{Ch2} and \cite[(4.5)]{BKD} involves residues of the spectral zeta function of the Laplace-Beltrami operator as well as the finite part of the spectral zeta function at a particular point $s\in\mathbb{C}$. In a more general setup in \cite[(7.22)]{BS2} the constant term is expressed as the infinite sum of residues of the spectral zeta function of a certain operator on $(N,g_N)$ plus the analytic continuation of the zeta function at a particular point. Here we show that the sum in the expression of the constant term is finite for the case of conic singularities. 

Since the manifold $(M,g)$ is non-compact, it may happen that the Laplace-Beltrami operator on $(M,g)$ has many self-adjoint extensions. To be able to apply the Singular Asymptotics Lemma we need the operator to satisfy the scaling property, see Section~\ref{section Laplace on infinite cone}. It is known that the Friedrichs extension has this property and we restrict our attention to this particular self-adjoint extension. We now present the main theorem.

\begin{theo}\label{main theorem}
Let $\Delta$ be the Laplace-Beltrami operator on smooth functions with compact support on $(M,g)$. If $m\geq4$, then $\Delta$ is essentially self-adjoint operator, otherwise we consider the Friedrichs extension of $\Delta$. Denote the self-adjoint extension of the Laplace-Beltrami operator by the same symbol~$\Delta$. Then
\begin{equation}\label{main expansion}
\tr e^{-t\Delta}\sim_{t\to0+}(4\pi t)^{-\frac{m}{2}}
\sum_{j=0}^{\infty}\tilde{a}_jt^j
+b
+c\log t,
\end{equation}
\begin{itemize}
\item[(a)]where
\begin{align*}
\tilde{a}_j=
\begin{cases}
\int_Mu_j\dvol_M \text{ for } j\leq m/2-1, \\
\fint_Mu_j\dvol_M \text{ for } j>m/2-1.
\end{cases}
\end{align*}
Above $\fint$ denotes the regularized integral, which we define in Section~\ref{regularized integrals}, of local quantities $u_j$ in (\ref{smooth terms}).
\item[(b)] The constant term $b$ in general cannot be written in terms of local quantities, and is given by
\begin{align*}
b
=&-\frac{1}{2}\Res_0\zeta^{\frac{m-2}{2}}_N(-1/2)
+\frac{\Gamma'(-\frac1 2)}{4\sqrt{\pi}}\Res_1\zeta^{\frac{m-2}{2}}_N(-1/2)\\
&-\frac1 4\sum_{1\leq j\leq m/2}j^{-1}B_{2j}\Res_1\zeta^{\frac{m-2}{2}}_N(j-1/2),
\end{align*}
where $\zeta^l_N(s)=\sum_{\lambda\in\Spec\Delta_N}(\lambda+l^2)^{-s}$ is the spectral zeta function shifted by $l$. The constants $B_{2j}$ are the Bernoulli numbers, $\Res_0f(s_0)$ is the regular analytic continuation of a function $f(s)$ at $s=s_0$, and $\Res_1f(s_0)$ is the residue of the function $f(s)$ at $s=s_0$.
\item[(c)] The logarithmic term is given by
\begin{equation}
c=
\begin{cases}
\frac{1}{2(4\pi)^{\frac{m}{2}}}\sum_{k=0}^{\frac{m}{2}}(-1)^{k+1}\frac{(m-2)^{2k}}{4^k k!}a^N_{\frac{m}{2}-k}, &\text{ for }m \text{ -- even},\\
0, &\text{ for }m \text{ -- odd}.
\end{cases}
\end{equation}
\item[(d)] If $c=0$ then $\tilde{a}_{m/2}=\int_Mu_{m/2}\dvol_M$ does not have a contribution from the singularity.
\end{itemize}
Here $\Delta_N$ is the Laplace-Beltrami operator on the cross-section $(N,g_N)$ and $a^N_j, j\geq0$ denote the coefficients in the heat trace expansion (\ref{smooth expansion}) on $(N,g_N)$.
\end{theo}

Above we need to regularize the integrals, because in general $\int_Mu_j\dvol_M$ diverges. If for some $j\geq0$ the integral converges, i.e. $\fint_Mu_j\dvol_M=\int_Mu_j\dvol_M$, then in this case $\tilde{a}_j$ is equal to $a_j$ from (\ref{smooth terms}).

Theorem~\ref{main theorem} allows to connect the coefficients in (\ref{main expansion}) to the geometry of $(M,g)$. It is now natural to pose the following question: given the coefficients in (\ref{main expansion}), what information about singularities can be obtained? The idea is to compare the expansion (\ref{main expansion}) to the expansion on a smooth compact manifold (\ref{smooth expansion}). This will be done in the subsequent paper.

This article is organized as follows. In Section~\ref{preliminaries}, we present a geometric setup, then define regularized integrals following \cite[Section~2.1]{L}, and state the main lemma from \cite[p.\,372]{BS2}, the Singular Asymptotics Lemma. In Section~\ref{computations}, we prove that the conditions of the Singular Asymptotics Lemma are satisfied in our case. We then apply it to the expansion of the trace of the resolvent. In Section~\ref{heat trace expansion}, we compute the coefficients in the heat trace expansion for a manifold with conic singularities. In Section~\ref{proof of main theorem}, we assemble the proof of Theorem~\ref{main theorem}.

\end{section}

\begin{section}{Preliminaries}\label{preliminaries}
\begin{subsection}{The Laplace-Beltrami operator on an open manifold}\label{geometric setup}

In this section we present some basic notions and theorems about operators on Hilbert spaces, following \cite{W}. Then we show how this results apply to the Laplace-Beltrami operator on a manifold with conic singularity.

Let $H_1$, $H_2$ be Hilbert spaces. Let $A$ be an operator from $H_1$ to $H_2$, and $B$ be an operator from $H_2$ to $H_1$.  The operator $B$ is called a {\itshape formal adjoint} of $A$ if we have
$$
\left<h,Ag\right>=\left<Bh,g\right> \text{ for all }g\in D(A), h\in D(B),
$$
where $D(A),D(B)$ are the domains of the operators $A,B$. We denote formal adjoint of $A$ by $A^{\dag}$.

Let $A$ be an operator on a Hilbert space $H$. The operator $A$ is called {\itshape symmetric} if for any elements $h,g$ from its domain we have $\left<h,Ag\right>=\left<Ah,g\right>$. A densely defined symmetric operator $A$ is called {\itshape self-adjoint} if it is equal to its adjoint $A=A^\dag$ and {\itshape essentially self-adjoint} if its closure is equal to its adjoint. An operator $B$ is called an {\itshape extension} of $A$ if we have 
$$
D(A)\subset D(B) \text{ and } Ah=Bh \text{ for } h\in D(A).
$$
If $A$ is a symmetric operator, then $A\subset A^\dag$, \cite[p.72]{W}. For every symmetric extension $B$ of $A$ we have $A\subset B\subset B^\dag\subset A^\dag$. If $B$ is self-adjoint, then $A\subset B=B^\dag\subset A^\dag$.

A symmetric operator $A$ on the Hilbert space $H$ is said to be {\itshape bounded from below} if there exists $a\in\mathbb{R}$ such that $\left<h,Ah\right>\geq a\Vert h\Vert^2$ for all $h\in D(A)$. Every $a$ of this kind is called a {\itshape lower bound}. If zero is a lower bound of $A$, then $A$ is called {\itshape non-negative}.

\begin{theo}[Friedrichs extension, {\cite[Theorem~5.38]{W}}]\label{exitence Friedrichs extension}
A non-negative densely defined symmetric operator $A$ on a Hilbert space $H$ has a non-negative self-adjoint extension.
\end{theo}

Let $(M,g)$ be a non-complete smooth Riemannian manifold that possesses a conic singularity. Consider the space of smooth functions with compact support $C^{\infty}_c(M)$ on $(M,g)$, and the space of differential one-forms with compact support $\lambda^1_c(M):=C^{\infty}_c(\Lambda T^*M)$. There are certain first order differential operators defined between these spaces, exterior derivative $d:C^{\infty}_c\to\lambda^1_c(M)$ and its formal adjoint $d^{\dag}=-*d*:\lambda^1_c(M)\to C^{\infty}_c(M)$, where $*$ is the Hodge-star operator on the differential forms on $(M,g)$. Furthermore,  the operator $\Delta:=d^{\dag}d$, defined on the smooth functions with compact support, is called the Laplace-Beltrami operator on $(M,g)$.

\begin{prop}
The operator $\Delta$ is densely defined in $L^2(M)$, symmetric and non-negative.
\end{prop}
\begin{proof}
Since $C^\infty_c(M)$ is dense in $L^2(M)$, the operators $d$ and $d^{\dag}$ are densely defined respectively in $L^2(M)$ and $L^2(\Lambda T^*M)$. Hence $\Delta$ is densely defined in $L^2(M)$. Let $f\in C^{\infty}_c(M)$, then
$$
(\Delta f,f)=(d^{\dag}df,f)=(df,df)=(f,\Delta f)\geq0.
$$
From this follows that $\Delta$ is symmetric and non-negative.
\end{proof}

By Theorem~\ref{exitence Friedrichs extension}, the operator $\Delta$ admits a self-adjoint extension. In Section~\ref{proof of main theorem}, we observe that for $\dim M=m<4$ there can be many self-adjoint extensions of $\Delta$, if this is the case, we choose the Friedrichs extension $\Delta^F$, which we denote simply by $\Delta$. The reason we choose the Friedrichs extension is that it satisfies the scaling property, Section~\ref{section Laplace on infinite cone}, which we need in Lemma~\ref{integrability condition}.

In the next sections, we discuss an asymptotic expansion of the heat trace of the Laplace-Beltrami operator. For this purpose we deal with the operator separately on a neighbourhood $(U,g_{\text{conic}})$ and on the regular part $(M\setminus U,g)$. For the restriction $\Delta|_{M\setminus U}$ of the Laplace-Beltrami operator $\Delta$ to the regular part, we use the methods applicable for a compact manifold. As for the restriction $\Delta|_U$, we first extend $(U,g_{\text{conic}})$ to an infinite cone $((0,+\infty)\times N,g_{\text{conic}})$, then extend the Laplace-Beltrami operator to the infinite cone and multiply the restriction $\Delta|_{(0,+\infty)\times N}$ by a function with the support near the tip of a cone and use the Singular Asymptotics Lemma, Lemma~\ref{SAL}. To glue the result on the infinite cone and the result on the regular part, we use a partition of unity. We observe that the heat trace expansion does not depend on $\varepsilon$ in (\ref{conic metric}).

\end{subsection}

\begin{subsection}{Spaces \texorpdfstring{$L^2((0,\varepsilon)\times N)$}{Lg} and \texorpdfstring{$L^2((0,\varepsilon),L^2(N))$}{Lg}}

In this subsection we introduce spaces that we consider in Section~\ref{computations}, and construct a bijective unitary map $(\ref{bijective map})$ between these spaces.

Let $E$ be a vector bundle over a smooth manifold $M$. Denote by \\*$C^\infty(M, E)$ space of smooth sections of $E$ over $M$.

Let $I:=(0,\varepsilon)$, where $0<\varepsilon\leq+\infty$, and let $X$ be any set. Define
\begin{equation}\label{smooth map}
C^\infty(I,X):=\{\varphi:I\to X\mid\varphi \text{ is smooth }\}.
\end{equation}

Consider a manifold $N$ and a projection map
$$
\pi_N:I\times N\to N.
$$

\begin{lemm}\label{lemma isomorphic spaces}
Let $G$ be a vector bundle over $N$. Then a section of pull-back bundle $\pi^*_NG$ at every $r\in I$ is a section of $G$, i.e. the following spaces are isomorphic
$$
C^{\infty}(I\times N,\pi^*_NG)
\simeq C^{\infty}(I,C^{\infty}(N,G)).
$$
\end{lemm}

\begin{proof}
We have the diagram
\begin{equation*}
\begin{tikzcd}
\pi_N^*G
\arrow[rightarrow]{r} \arrow[rightarrow]{d} 
&G \arrow[rightarrow]{d}\\
 I\times N
\arrow[rightarrow]{r}{\pi_N}
  & N
\end{tikzcd}
\end{equation*}
Let $(U_i,\tau_{ij})$ be a covering of $N$ by open sets $U_i$ such that the bundle $G$ restricted to $U_i$ is trivial $G|_{U_i}\simeq U_i\times\mathbb{R}^k$. Maps $\tau_{ij}$ are corresponding transition maps, i.e. smooth maps $\tau_{ij}:U_i\cap U_j\to GL(k)$, where $k$ is the rank of the bundle $G$. Then $(I\times U_i,\tau_{ij}\circ\pi_N)$ is a covering for the pull-back bundle $\pi_N^*G$.

Let $s\in C^{\infty}(I\times N,\pi^*_NG)$. Restrict the section $s$ on a chart 
$$
s|_{I\times U_i}:I\times U_i\to I\times U_i\times\mathbb{R}^k.
$$
Since $C^{\infty}(I\times U_i,I\times U_i\times\mathbb{R}^k)\simeq C^{\infty}(I\times U_i,\mathbb{R}^k)$, we may reason in terms of maps. By the exponential law for smooth maps \cite[Theorem 3.12, Corollary 3.13]{KM}, we have
$$
C^{\infty}(I\times U_i,\mathbb{R}^k)\simeq C^{\infty}(I,C^{\infty}(U_i,\mathbb{R}^k)).
$$
Above $C^{\infty}(\cdot,\cdot)$ denotes a space of the smooth maps as in (\ref{smooth map}).
Now we pass back to the space of the smooth sections
$$
C^{\infty}(I,C^{\infty}(U_i,\mathbb{R}^k))\simeq C^{\infty}(I,C^{\infty}(U_i,U_i\times\mathbb{R}^k)),
$$
and obtain the isomorphism of spaces
$$
C^{\infty}(I\times U_i,I\times U_i\times\mathbb{R}^k)
\simeq C^{\infty}(I,C^{\infty}(U_i,U_i\times\mathbb{R}^k)).
$$
This isomorphism holds for any chart. Since the transition maps are smooth, we obtain the desired isomorphism.
\end{proof}

Let $(N,g_N)$ be a closed smooth Riemannian manifold. Denote space of the differential $k$-forms on $(N,g_N)$ by $\lambda^k(N)$.

\begin{lemm}
The following spaces are isomorphic
$$
\lambda^k(I\times N)\simeq C^{\infty}(I,\lambda^k(N)\oplus\lambda^{k-1}(N)),
$$
in particular
$$
C^\infty(I\times N)\simeq C^\infty(I,C^\infty(N)).
$$
\end{lemm}

\begin{proof}
Define
$$
G:=\Lambda^kT^*N\oplus\Lambda^{k-1}T^*N.
$$
Then $\pi^*_NG=\Lambda^kT^*N\oplus\Lambda^{k-1}T^*N$, i.e. at a point $(r,p)\in I\times N$ the fiber of $\pi^*_NG$ is $\Lambda^kT_p^*N\oplus\Lambda^{k-1}T_p^*N$ for every $r\in I$. Note also that
\begin{equation*}
\begin{split}
\Lambda^k(T_{(r,p)}^*(I\times N))
&=\Lambda^k(T_r^*I\oplus T_p^*N)\\
&=(\Lambda^0T_r^*I\otimes\Lambda^kT_p^*N)
\oplus(\Lambda^1T_r^*I\otimes\Lambda^{k-1}T_p^*N)\\
&=\Lambda^kT_p^*N\oplus\Lambda^{k-1}T_p^*N.
\end{split}
\end{equation*}
Hence $\Lambda^kT^*(I\times N)=\pi^*_NG$.

By Lemma \ref{lemma isomorphic spaces}, we obtain
$$
C^{\infty}(I\times N, \Lambda^kT^*(I\times N))
\simeq
C^{\infty}(I,C^{\infty}(N, \Lambda^kT^*N\oplus\Lambda^{k-1}T^*N)).
$$

\end{proof}

Define by $L^2(I\times N)$ the Hilbert space of the square-integrable functions on $I\times N$ with the inner product
$$
\left<\varphi,\psi\right>_{L^2(I\times N)}=\int_I\int_N\varphi\psi r^{\dim N} \dvol_N dr,
$$
where $\varphi,\psi\in L^2(I\times N)$. Define by $L^2(I,L^2(N)):=\{\varphi:I\to L^2(N)\mid\varphi \text{ is square integrable }\}$ the Hilbert space with the inner product
$$
\left<\varphi,\psi\right>_{L^2(I,L^2(N))}=\int_I\int_N\varphi\psi \dvol_N dr,
$$
where $\varphi,\psi\in L^2(I,L^2(N))$.
Then there is a bijective unitary map
\begin{equation}\label{bijective map}
\Psi:L^2(I,L^2(N)) \to L^2(I\times N).
\end{equation}
For $\varphi\in L^2(I,L^2(N))$ the map is defined by
$$
\varphi \mapsto r^{-\frac{\dim N}{2}}\varphi.
$$
\end{subsection}

\begin{subsection}{Trace lemma}

 The Trace Lemma will be used for the proofs in Section~\ref{computations}. The proof of Lemma can be found in \cite[Appendix A]{BS2}.

Let $H$ be a Hilbert space. Denote by $C_1(H)$ the trace class operators, i.e. the first Schatten class of operators. Denote by $||\cdot||_{\tr}$ and $||\cdot||_{HS}$ respectively trace norm and Hilbert-Schmidt operator norm.

\begin{lemm}[Trace Lemma]\label{trace_lemma}
Let $T$ be a trace class operator on $L^2(\mathbb{R},H)$. Then $T$ has a kernel $t(x,y)$, so that for $u(x)\in dom(T)$ we have $Tu(x)=\int_{-\infty}^{\infty}t(x,y)u(y)dy$ and
$$
h\mapsto t(\cdot,\cdot+h)
$$
is a continuous map from $\mathbb{R}$ into $L^1(\mathbb{R},C_1(H))$. Furthermore,
$$
\int_{-\infty}^{\infty}||t(x,x)||_{\tr}dx\leq||T||_{\tr}
$$
and
$$
\int_{-\infty}^{\infty}\tr(t(x,x))dx=\tr T.
$$
\end{lemm}

\end{subsection}

\begin{subsection}{Regularized integrals and the Singular Asymptotics Lemma}\label{regularized integrals}

In this section we define the regularized integral over the interval $(0,\infty)$ for a certain class of locally integrable functions using the Mellin transform. We follow \cite[Section~2.1]{L}. First, we recall the definition of the Mellin transform and specify the class of functions with which we will work.

\begin{defi}
Let $H$ be a Hilbert space. For a function $f\in C^{\infty}_c((0,\infty),H)$, the {\itshape Mellin transform} is defined by
$$
Mf(s):=\int_0^{\infty}x^{s-1}f(x)dx,
$$
for $s\in\mathbb{C}$.
\end{defi}

Let $p,q>0$ and denote $L_{loc}^1(0,\infty):=L_{loc}^1((0,\infty))$.
\begin{defi}\label{p,q-functions}
Let $f\in L_{loc}^1(0,\infty)$ be a locally integrable function such that
\begin{align*}
f(x)
&=\sum_{j=1}^N\sum_{k=0}^{m_j}a_{jk}x^{\alpha_j}\log^kx+x^pf_1(x)\\
&=\sum_{j=1}^M\sum_{k=0}^{m'_j}b_{jk}x^{\beta_j}\log^kx+x^{-q}f_2(x),
\end{align*}
where $f_1\in L_{loc}^1([0,\infty))$, $f_2\in L^1([1,\infty))$ and $\alpha_j,\beta_j\in\mathbb{C}$ with real parts $\re(\alpha_j)\leq p-1$ increasing and $\re(\beta_j)\geq-q-1$ decreasing as $j$ grows.
Denote the class of such functions by $L_{p,q}(0,\infty)\subset L^1_{loc}(0,\infty)$.
\end{defi}

Also denote
\begin{align*}
&L_{\infty,q}(0,\infty):=\cap_{p>0}L_{p,q}(0,\infty),\\
&L_{p,\infty}(0,\infty):=\cap_{q>0}L_{p,q}(0,\infty),\\
&L_{as}(0,\infty):=L_{\infty,\infty}(0,\infty):=\cap_{p>0}L_{p,\infty}(0,\infty).
\end{align*}

\begin{rema}
In Definition \ref{p,q-functions} the first equality reflects the behaviour of $f(x)$ as $x\to0$ and the second equality reflects the behaviour of $f(x)$ as $x\to\infty$.
\end{rema}

\begin{rema}
For $f\in L_{p,q}(0,\infty)$ and $\re(s)>-\min_{1\leq j\leq N}\{\re(\alpha_j)\}$, the function $x^{s-1}f(x)$ is locally integrable with respect to $x\in[0,\infty)$. 
\end{rema}

We extend the Mellin transform to $f\in L_{p,q}(0,\infty)$ by splitting it into two integrals. For $c>0$, denote
$$
(Mf)(s):=(M_{[0,c]}f)(s)+(M_{[c,\infty]}f)(s):=\int_0^cx^{s-1}f(x)dx+\int_c^{\infty}x^{s-1}f(x)dx.
$$
The next proposition shows that the Mellin transform is well defined, for the proof see \cite[Section~2.1]{L}.

\begin{prop}
Let $p,q>0$, $f\in L_{pq}(0,\infty)$ and $s\in\mathbb{C}$, such that $1-p<\re(s)<1+q$. Then 
$$
(Mf)(s)=(M_{[0,c]}f)(s)+(M_{[c,\infty]}f)(s)
$$
is a meromorphic function in a strip $1-p<\re(s)<1+q$ and is independent of $c$. Moreover, the continuation of $(Mf)(s)$ may have poles at most of order $m_j+1$ at $s=-\alpha_j$ and $m'_j+1$ at $s=-\beta_j$ in the notations of Definition \ref{p,q-functions}.
\end{prop}

Let $f$ be a meromorphic function. Denote by $\Res_kf(z_0)$ the coefficient of $(z-z_0)^{-k}$ in the Laurent expansion of $f$ near $z_0$
$$
f(z)=\sum_{k=-m}^{\infty}\Res_{-k}f(z_0)(z-z_0)^k.
$$

\begin{defi}\label{definition of regularized integral}
Let $f\in L_{p,q}(0,\infty)$. A {\itshape regularized integral} is the constant coefficient in the Laurent expansion near $s=1$ of the Mellin transform of $f$, i.e.
$$
\fint_0^{\infty}f(x)dx:=\Res_0(Mf)(1).
$$
\end{defi}


Now we are ready to state the Singular Asymptotics Lemma.

Let $C:=\{|\arg\zeta|<\pi-\epsilon\}$ for some $\epsilon>0$.

\clearpage
\begin{lemm}[Singular Asymptotics Lemma, {\cite[p.372]{BS2}}]\label{SAL}
Let $\sigma(r,\zeta)$ be defined on $\mathbb{R}\times C$ and satisfy the following conditions
\begin{itemize}
\item [(1)] $\sigma(r,\zeta)$ is $C^{\infty}$ with respect to $r$ and has analytic derivatives with respect to $\zeta$;
\item [(2)] there exist Schwartz functions $\sigma_{\alpha j}(r)\in \mathscr{S}(\mathbb{R})$ such that for $|\zeta|\geq1$ and $0\leq r\leq|\zeta|/C_0$,
$$
\left|r^J\partial_r^K\left(\sigma(r,\zeta)-\sum_{Re\alpha\geq-M}\sum_{j=0}^{J_\alpha}\sigma_{\alpha j}(r)\zeta^{\alpha}\log^j\zeta\right)\right|
\leq C_{JKM}|\zeta|^{-M};
$$
\item[(3)](integrability condition) the derivatives $\sigma^{(j)}(r,\zeta):=\partial^j_r\sigma(r,\zeta)$ satisfy uniformly for $0\leq t\leq1$ and $|\xi|=C_0$
$$
\int_0^1\int_0^1s^j|\sigma^{(j)}(st,s\xi)|dsdt\leq C_j.
$$
\end{itemize}
Then
\begin{align*}
\int_0^{\infty}\sigma(r,rz)dr\sim_{z\to\infty}
&\sum_{k=0}^{\infty}z^{-k-1}\fint_0^{\infty}\frac{\zeta^k}{k!}\sigma^{(k)}(0,\zeta)d\zeta\\
+&\sum_{\alpha}\sum_{j=0}^{J_{\alpha}}\fint_0^{\infty}\sigma_{\alpha j}(r)(rz)^{\alpha}\log^j(rz)dr\\
+&\sum_{\alpha=-1}^{-\infty}\sum_{j=0}^{J_{\alpha}}\sigma^{(-\alpha-1)}_{\alpha j}(0)\frac{z^{\alpha}\log^{j+1}z}{(j+1)(-\alpha-1)!}.
\end{align*}

\end{lemm}

\begin{rema}
Above $\alpha$ is any sequence of complex numbers with $\re(\alpha)\to-\infty$. The last sum in the expansion includes only those $\alpha$ that happen to be negative integers. $J_\alpha$ is the biggest power of $\log\zeta$ that occurs for $\alpha$.
\end{rema}
\end{subsection}

\end{section}

\begin{section}{Local computations}\label{computations}

The aim of this chapter is to prove Theorem~\ref{main theorem}. First, we recall the asymptotic expansion of the heat kernel along the diagonal, which is a local result and does not require completeness of the manifold. The expansion is given in terms of the curvature tensor and its covariant derivatives. In the case of a compact manifold, one integrates the terms in the local expansion over the manifold and obtains the classical heat kernel expansion (\ref{smooth expansion}). In the case of a non-complete manifold $(M,g)$ with conic singularities, defined in Section~\ref{geometric setup}, we compute the curvature tensor near the conic point and observe that the integrals over the manifold in general diverge near the conic point. Then we use the Singular Asymptotics Lemma to obtain the heat trace expansion from the local heat kernel expansion.

\begin{subsection}{Local expansion of the heat kernel}\label{section local expansion}

Let $(M,g)$ be a Riemannian manifold, possibly non-complete, and $\Delta$ be the Laplace-Beltrami operator on $(M,g)$. For $(p,q)\in M\times M$ denote the heat kernel by $e^{-t\Delta}(p,q)$. The heat kernel along the diagonal $(p,p)\in M\times M$ is denoted by $e^{-t\Delta}(p)$. The next proposition gives an expansion of the heat kernel along the diagonal on any compact subset of $(M,g)$.

\begin{theo}[{\cite[Section III.E]{BGM}} ] \label{local_heat_expansion}
Let $K\subset M$ be any compact set and $p\in K$. There is an asymptotic expansion of the heat kernel along the diagonal
$$
\lVert e^{-t\Delta}(p)-(4\pi t)^{-\frac{\dim M}{2}}\sum_{i=0}^{j}t^iu_i(p)\rVert\leq C_j(K)t^{j+1},
$$
where $C_j(K)$ is some constant which depends on the compact set $K$. Moreover, $u_0(p)\equiv1$ and $u_1(p)=\frac16\Scal(p)$, where $\Scal(p)$ is the scalar curvature at $p\in M$, and all $u_i(p)$ are polynomials on the curvature tensor and its covariant derivatives.
\end{theo}

Furthermore \cite[p.201 Theorem 3.3.1]{G2}
\begin{equation}\label{term u_2}
u_2(p)=\frac{1}{360}\left(12\Delta \Scal(p)+5\Scal(p)^2-2\lvert\Ric(p)\rvert^2+2\lvert\R(p)\rvert^2\right),
\end{equation}
where
$$
\lvert\R(p)\rvert^2:=\R_{ijkl}(p)\R_{ijkl}(p)g^{ii}(p)g^{jj}(p)g^{kk}(p)g^{ll}(p),
$$
$$
\lvert\Ric(p)\rvert^2:=\Ric_{ij}(p)\Ric_{ij}(p)g^{ii}(p)g^{jj}(p).
$$
Above, $\R_{ijkl}(p)$ is the Riemann curvature tensor, $\Ric_{ij}(p)$ is the Ricci tensor, $\Scal(p)$ is the scalar curvature.

The heat operator is closely related to the resolvent operator by the Cauchy's differentiation formula. For a positively oriented closed path $\gamma$ in the complex plane surrounding the spectrum of $\Delta$ and for $d\in\mathbb{N}$, we have

\begin{align}\label{Cauchy formula}
e^{-t\Delta}=-t^{1-d}\frac{(d-1)!}{2\pi i}\int_{\gamma}e^{-t\mu}(\Delta-\mu)^{-d}d\mu.
\end{align}

To interpolate between the expansion of the heat trace and the expansion of the resolvent trace we will use the following formulas

\begin{equation}\label{integral1}
\begin{split}
\int_{\gamma}e^{-t\mu}(-\mu)^{-n}d\mu
&=2\pi i\Res_{\mu=0}(e^{-t\mu}(-\mu)^{-n})\\
&=\frac{2\pi i}{(n-1)!}\lim_{\mu\to0}\left(\frac{d}{d\mu}\right)^{n-1}(-1)^ne^{-t\mu}
=-\frac{2\pi i}{\Gamma(n)}t^{n-1}
\end{split}
\end{equation}

and

\begin{equation}\label{integral2}
\begin{split}
\int_{\gamma}e^{-t\mu}(-\mu)^{-n}\log(-\mu)d\mu
&=-\frac{d}{dn}\int_{\gamma}e^{-t\mu}(-\mu)^{-n}d\mu
=\frac{d}{dn}\left(\frac{2\pi i}{\Gamma(n)}t^{n-1}\right)\\
&=2\pi i\frac{t^{n-1}\log t\Gamma(n)-t^{n-1}\Gamma(n)'}{\Gamma(n)^2}\\
&=\frac{2\pi i}{\Gamma(n)}t^{n-1}\log t-\frac{2\pi i}{\Gamma(n)}\frac{\Gamma'(n)}{\Gamma(n)}t^{n-1}.
\end{split}
\end{equation}

On $(M,g)$ by Theorem \ref{local_heat_expansion}, we have the local asymptotic expansion of the heat kernel along the diagonal. Then using Cauchy's differentiation formula (\ref{Cauchy formula}), (\ref{integral1}) and (\ref{integral2}), we obtain the expansion of the kernel of the resolvent along the diagonal for $p\in K\subset M$. Denote $z^2:=-\mu$. By \cite[p.61, Lemma 1.7.2]{G},

\begin{equation*}
\begin{split}
&\left\Vert(\Delta+z^2)^{-d}(p)-
(4\pi)^{-\frac{m}{2}}\sum_{j=d-\frac{m}{2}}^{k}z^{-2j}u_{j+\frac{m}{2}-d}(p)\frac{\Gamma(j)}{(d-1)!}\right\Vert
\leq \tilde{C}_k(K)z^{-2j-2},
\end{split}
\end{equation*}
for some $\tilde{C}(K)>0$. For convenience denote $l:=j+\frac{m}{2}-d$, then

\begin{equation}\label{local_resolvent_expansion}
\begin{split}
&\left\Vert(\Delta+z^2)^{-d}(p)-
(4\pi)^{-\frac{m}{2}}\sum_{l=0}^{k+m/2-d}z^{-2d+m-2l}u_l(p)\frac{\Gamma(-\frac{m}{2}+d+l)}{(d-1)!}\right\Vert\\
&\leq \tilde{C}_k(K)z^{-2d+m-2l-2},
\end{split}
\end{equation}
where $u_0(p)\equiv1$ and $u_1(p)=\frac{\Scal(p)}{6}$.

\end{subsection}

\begin{subsection}{Curvature tensor in polar coordinates}
In this section we give explicit formulas for the curvature tensors in the neighbourhood $(U,g_{\text{conic}})$ of the conic singularity in terms of the curvature tensors on the cross-section manifold $(N,g_N)$ of dimension $n$. Let $x=(x^1,\dots,x^n)$ be local coordinates on $(N,g_N)$ and $p=(r,x^1,\dots,x^n)\in U$. For $i,j\in\{0,1,\dots,n\}$ denote by $\tilde{g}_{ij}$ the components of the metric tensor $g_{\text{conic}}=dr^2+r^2g_N$, and by $g_{ij}$ for $i,j\in\{1,\dots,n\}$ the components of the metric tensor $g_N$. Then
$$
\tilde{g}_{00}=1,  \;\;\;\;  \tilde{g}_{i0}=\tilde{g}_{0i}=0 \text{, for } i>0,
$$
and
\begin{align*}
\tilde{g}_{ij}=r^2g_{ij}, \text{ for } i,j>0.
\end{align*}

We use the standard notations for the tensors that correspond to the metric $g_{ij}$. For tensors corresponding to the metric $\tilde{g}_{ij}$, we use the same notations, but with tildes. To stress that the tensor depends on a point we use $\tilde{g}_{ij}(p)=\tilde{g}_{ij}(r,x)$. If it is clear, we may omit a point $p$ to simplify the notations. Denote the derivative of $\tilde{g}_{ij}$ with respect to the $k$-th coordinate by $\tilde{g}_{ij,k}$. If $i$ or $j$ or both are equal to zero then $\tilde{g}_{ij,k}=0$ for any $k\in{0,1,\dots,n}$. Suppose $i,j\neq0$, then

\begin{align*}
\tilde{g}_{ij,k}(r,x)=
\begin{cases}
2rg_{ij}(x), \text{ if } k=0,\\
r^2g_{ij,k}(x), \text{ if } k\neq0.
\end{cases}
\end{align*}

The Christoffel symbols are of course
$$
\tilde{\Gamma}^i_{jk}=\frac12 \tilde{g}^{im}\left(\tilde{g}_{mj,k}+\tilde{g}_{mk,j}-\tilde{g}_{jk,m}\right),
$$
but now we express them in terms of the Christoffel symbols $\Gamma^i_{jk}$ and the metric tensor $g_{ij}$.

Let $i=0$
\begin{align*}
\tilde{\Gamma}^0_{jk}=
\begin{cases}
0, \text{ if } j=0 \text{ or } k=0,\\
-rg_{jk}, \text{ otherwise }.
\end{cases}
\end{align*}

Assume $i\neq0$ and let $j=0$, then
\begin{align*}
\tilde{\Gamma}^i_{0k}=\tilde{\Gamma}^i_{k0}
&=\frac12 \tilde{g}^{im}\left(\tilde{g}_{m0,k}+\tilde{g}_{mk,0}-\tilde{g}_{0k,m}\right)\\
&=\frac12 r^{-2}g^{im}\left(\tilde{g}_{mk,0}\right)\\
&=\frac12 r^{-2}g^{im}\left(2rg_{mk}\right)\\
&=r^{-1}\delta^i_k.
\end{align*}
If both $j=k=0$, then $\tilde{\Gamma}^i_{00}=0$.
Assume that $i,j$ and $k$ are all non-zero. Then

$$
\tilde{\Gamma}^i_{jk}=\Gamma^i_{jk}.
$$


The scalar curvature $\tilde{\Scal}$ can be expressed in terms of the Christoffel symbols $\tilde{\Gamma}^i_{jk}$ in the following way

\begin{align*}
\tilde{\Scal}=&\tilde{g}^{ij}\left(\tilde{\Gamma}^m_{ij,m}-\tilde{\Gamma}^m_{im,j}
+\tilde{\Gamma}^l_{ij}\tilde{\Gamma}^m_{ml}-\tilde{\Gamma}^l_{im}\tilde{\Gamma}^m_{jl}\right)\\
=&\tilde{g}^{0j}\left(\tilde{\Gamma}^m_{0j,m}-\tilde{\Gamma}^m_{0m,j}
+\tilde{\Gamma}^l_{0j}\tilde{\Gamma}^m_{ml}-\tilde{\Gamma}^l_{0m}\tilde{\Gamma}^m_{jl}\right)\\
&+\sum_{i\neq0}\tilde{g}^{ij}\left(\tilde{\Gamma}^m_{ij,m}-\tilde{\Gamma}^m_{im,j}+\tilde{\Gamma}^l_{ij}\tilde{\Gamma}^m_{ml}-\tilde{\Gamma}^l_{im}\tilde{\Gamma}^m_{jl}\right)\\
=:&\tilde{\Scal}_0+\tilde{\Scal}_1.
\end{align*}
Now we compute $\tilde{\Scal}_0$ and $\tilde{\Scal}_1$.

Since $\tilde{g}^{0j}$ is equal to zero for any $j$, $j\neq0$, we have 

\begin{align*}
\tilde{\Scal}_0=&\left(\tilde{\Gamma}^m_{00,m}-\tilde{\Gamma}^m_{0m,0}
+\tilde{\Gamma}^l_{00}\tilde{\Gamma}^m_{ml}-\tilde{\Gamma}^l_{0m}\tilde{\Gamma}^m_{0l}\right)\\
=&-\partial_m(r^{-1}\delta^m_m)-r^{-1}\delta^l_mr^{-1}\delta^m_l=r^{-2}n-r^{-2}n^2=-r^{-2}n(n-1)
\end{align*}
and

\begin{align*}
\tilde{\Scal}_1
=&\sum_{i\neq0}\tilde{g}^{ij}\left(\tilde{\Gamma}^m_{ij,m}-\tilde{\Gamma}^m_{im,j}+\tilde{\Gamma}^l_{ij}\tilde{\Gamma}^m_{ml}-\tilde{\Gamma}^l_{im}\tilde{\Gamma}^m_{jl}\right)\\
=&r^{-2}g^{ij}\left(-g_{ij}+\sum_{m\neq0}[\tilde{\Gamma}^m_{ij,m}
-\tilde{\Gamma}^m_{im,j}]
-g_{ij}\delta^m_m
+\sum_{l\neq0}\tilde{\Gamma}^l_{ij}\tilde{\Gamma}^m_{ml}
+g_{im}\delta^m_j\right.\\
&\left.-\sum_{l\neq0}\tilde{\Gamma}^l_{im}\tilde{\Gamma}^m_{jl}
+\delta^l_ig_{jl}\right)\\
=&r^{-2}g^{ij}\left(-g_{ij}+[\Gamma^m_{ij,m}
-\Gamma^m_{im,j}]
-g_{ij}n
+\Gamma^l_{ij}\Gamma^m_{ml}
+g_{ij}n
-\Gamma^l_{im}\Gamma^m_{jl}
+g_{ij}\right)\\
=&r^{-2}g^{ij}\left(\Gamma^m_{ij,m}
-\Gamma^m_{im,j}+\Gamma^l_{ij}\Gamma^m_{ml}
-\Gamma^l_{im}\Gamma^m_{jl}\right)=r^{-2}\Scal,
\end{align*}
where $\Scal$ is the scalar curvature at $x\in N$.

Therefore the scalar curvature in the polar coordinates $p=(r,x^1,\dots,x^n)$ is

\begin{equation}\label{scalar curvature}
\tilde{\Scal}(p)=r^{-2}\big(\Scal(x)-n(n-1)\big),
\end{equation}
where $p=(r,x)\in U$ and $x\in N$ and $\tilde{\Scal}(p)$ is the scalar curvature on $(U,g_{\text{conic}})$ and $\Scal(x)$ is the scalar curvature on $(N,g_N)$.


Recall that the Riemann curvature tensor and the Ricci tensor can be written using the Christoffel symbols as follows
$$
\R^i_{jkl}=\partial_k\Gamma^i_{jl}+\Gamma^i_{kp}\Gamma^p_{jl}-\partial_l\Gamma^i_{jk}-\Gamma^i_{lp}\Gamma^p_{jk}
$$
and
$$
\Ric_{ij}=\partial_m\Gamma^m_{ij}+\Gamma^m_{mp}\Gamma^p_{ij}-\partial_j\Gamma^m_{mi}-\Gamma^m_{jp}\Gamma^p_{im}.
$$
Now we express the Riemann tensor $\tilde{\R}^i_{jkl}$ in terms of the Riemann curvature tensor $\R^i_{jkl}$.

Let $i=0$ and $i,j,k$ be nonzero, then
$$
\tilde{\R}^0_{jkl}=0,
$$
also
$$
\tilde{\R}^i_{0kl}=\tilde{\R}^i_{j0l}=\tilde{\R}^i_{j00}=\tilde{\R}^i_{000}=0
$$
and
$$
\tilde{\R}^i_{00l}=\tilde{\R}^i_{0l0}=-r^{-2}\delta^i_l.
$$

If none of the indices $i,j,k,l$ is zero, we obtain
\begin{equation}\label{Riemann curvature}
\tilde{\R}_{ijkl}(p)=r^{-2}\big(\R_{ijkl}(x)-g_{ip}(x)g_{jm}(x)(\delta^p_k\delta^m_l-\delta^p_l\delta^m_k)\big).
\end{equation}

Similarly, for the tensor Ricci
\begin{equation}\label{Ricci curvature}
\tilde{\Ric}_{ij}(p)=r^{-2}\left(\Ric_{ij}(x)-(n-1)g_{ij}(x)\right).
\end{equation}

\end{subsection}

\begin{subsection}{The Laplace operator on the infinite cone}\label{section Laplace on infinite cone}
In this section we obtain an expression (\ref{Laplace on infinite cone}) of the Laplace-Beltrami operator on an infinite cone, and show that it satisfies the scaling property. 

The Laplace-Beltrami operator $\Delta:=d^{\dag}d$ acting on $C^{\infty}_c((0,\varepsilon)\times N)$ in the Hilbert space $L^2((0,\infty)\times N)$ with measure $r^{n}\dvol_Ndr$ can be written in terms of partial derivatives with respect to the local coordinates. We have $d^{\dag}=-*d*$, where $*$ is the Hodge-star operator on $(U,g_{\text{conic}})$.

Let $p=(r,x^1,\dots,x^n)\in U$ and $f(r,x^1\dots,x^n)\in C^{\infty}_c(U)$. Denote $g:=\det g_N=\det(g_{ij})$ and $(g^{ij})=(g_{ij})^{-1}$. Some computations that will be used later are given in the next proposition. Choose coordinates on $(N,g_N)$ such that the metric is diagonal.

\begin{prop}\label{proposition computations}
Let $A\in C^{\infty}_c(U)$. Then
\begin{itemize}
\item[a)]$*(Adr)=Ar^ng^{1/2}dx^1\wedge\dots\wedge dx^n$;
\item[b)]$*(Adx^i)=A(-1)^ir^{n-2}g^{ii}g^{1/2}dr\wedge dx^1\wedge\dots\hat{dx^i}\dots\wedge dx^n$;
\item[c)]$*(Adr\wedge dx^1\wedge\dots\wedge dx^n)=Ar^{-n}g^{-1/2}$.
\end{itemize}
\end{prop}

\begin{proof}
Let $X\in C^{\infty}_c(U)$.

a) Let $X$ be such that $*(Adr)=Xdx^1\wedge\dots\wedge dx^n$, then by the definition of the Hodge star operator, for any one-form $\alpha dr$ we have
$$
\alpha dr\wedge Xdx^1\wedge\dots\wedge dx^n
=\alpha Ar^ng^{1/2}dr\wedge dx^1\wedge\dots\wedge dx^n,
$$
hence $X=Ar^ng^{1/2}$.

b) Let $X$ be such that $*(Adx^i)=Xdr\wedge dx^1\wedge\dots\hat{dx^i}\dots\wedge dx^n$. Then for any one-form $\alpha dx^i$
$$
\alpha dx^i\wedge Xdr\wedge dx^1\wedge\dots\hat{dx^i}\dots\wedge dx^n
=\alpha Ar^{-2}g^{ii}r^ng^{1/2}dr\wedge dx^1\wedge\dots\wedge dx^n,
$$
hence 
$$
(-1)^i\alpha X=\alpha Ag^{ii}r^{n-2}g^{1/2}.
$$
Consequently, $X=(-1)^i Ag^{ii}r^{n-2}g^{1/2}$.

c) Let $X$ be such that $*(Adr\wedge dx^1\wedge\dots\wedge dx^n)=X$, then for any $n+1$-form $\alpha dr\wedge dx^1\wedge\dots\wedge dx^n$
$$
\alpha dr\wedge dx^1\wedge\dots\wedge dx^n X
=\alpha Ar^{-2n} g^{-1}r^ng^{1/2}dr\wedge dx^1\wedge\dots\wedge dx^n.
$$
Hence $X=Ar^{-n}g^{-1/2}$.
\end{proof}

Let $f(r,x^1,\dots,x^n)$ be a smooth function with compact support in $(U,g_{\text{conic}})$ and apply the operator $\Delta=-*d*d$ to it. By Proposition \ref{proposition computations} we obtain

\begin{equation}\label{laplace1}
\begin{split}
f(r,x^1,\dots,x^n)
\stackrel{d}{\mapsto}&
\partial_rfdr+\sum_{i=1}^n\partial_{x^i}fdx^i\\
\stackrel{*}{\mapsto}&
r^ng^{1/2}\partial_rfdx^1\wedge\dots\wedge dx^n\\
&+\sum_{i=1}^n(-1)^ir^{n-2}g^{ii}g^{1/2}\partial_{x^i}fdr\wedge dx^1\wedge\dots\hat{dx^i}\dots\wedge dx^n\\
\stackrel{d}{\mapsto}&
dr\wedge dx^1\wedge\dots\wedge dx^n\bigg(r^ng^{1/2}\partial_r^2f+nr^{n-1}g^{1/2}\partial_rf\\
&+\sum_{i=1}^n\big(r^{n-2}g^{ii}g^{1/2}\partial^2_{x^i}f
+r^{n-2}g^{1/2}\partial_{x^i}f\partial_{x^i}g^{ii}\bigg.\\
&+\bigg.\frac12r^{n-2}g^{-1/2}g^{ii}\partial_{x^i}f\partial_{x^i}g\big)\bigg)\\
\stackrel{*}{\mapsto}&
\partial_r^2f+nr^{-1}\partial_rf\\
&+\sum_{i=1}^n\big(r^{-2}g^{ii}\partial^2_{x^i}f
+r^{-2}\partial_{x^i}f\partial_{x^i}g^{ii}
+\frac12r^{-2}g^{-1}g^{ii}\partial_{x^i}f\partial_{x^i}g\big)\\
\stackrel{-}{\mapsto}&
-\partial_r^2f-nr^{-1}\partial_rf\\
&-r^{-2}\sum_{i=1}^n\big(g^{ii}\partial^2_{x^i}f
+\partial_{x^i}f\partial_{x^i}g^{ii}
+\frac12g^{-1}g^{ii}\partial_{x^i}f\partial_{x^i}g\big).
\end{split}
\end{equation}

Now consider the Laplace-Beltrami operator $\Delta_N$ on $(N,g_N)$
\begin{equation}\label{laplace2}
\begin{split}
\Delta_Nf=&-*d*df=-*d*\sum_{i=1}^n\partial_{x^i}fdx^i\\
=&-*d\left(\sum_{i=1}^n(-1)^ig^{1/2}g^{ii}\partial_{x^i}fdx^1\wedge\dots\hat{x^i}\dots\wedge dx^n\right)\\
=&-*\bigg(g^{1/2}g^{ii}\partial_{x^i}^2f
+g^{1/2}\partial_{x^i}f\partial_{x^i}g^{ii}\\
&+\frac12g^{-1/2}g^{ii}\partial_{x^i}f\partial_{x^i}g\bigg)dx^1\wedge\dots\wedge dx^n\\
=&-\left(g^{ii}\partial_{x^i}^2f
+\partial_{x^i}f\partial_{x^i}g^{ii}
+\frac12g^{-1}g^{ii}\partial_{x^i}f\partial_{x^i}g\right).
\end{split}
\end{equation}
By (\ref{laplace1}) and (\ref{laplace2})
$$
\Delta=-\partial_r^2-nr^{-1}\partial_r+r^{-2}\Delta_N.
$$

In (\ref{bijective map}), we defined a unitary map between the Hilbert spaces

\begin{equation*}
\Psi:L^2((0,\infty),L^2(N)) \to L^2((0,\infty)\times N),
\end{equation*}
$$
\Psi: f\mapsto r^{-n/2}f
$$
for $f\in L^2((0,\infty),L^2(N))$.

Define the operator 
$$
T:=\Psi^{-1}\Delta\Psi,
$$
acting on $C^{\infty}_c((0,\infty),C^{\infty}(N))$ in the Hilbert space $L^2((0,\infty),L^2(N))$.

Then
\begin{equation*}
\begin{split}
Tf=&r^{n/2}\Delta r^{-n/2}f\\
=&r^{n/2}\left(-\partial_r^2-nr^{-1}\partial_r+r^{-2}\Delta_N\right)r^{-n/2}f\\
=&-r^{n/2}\partial_r\left(r^{-n/2}\partial_rf-\frac{n}{2}r^{-n/2-1}f\right)
-nr^{n/2-1}\left(r^{-n/2}\partial_rf-\frac{n}{2}r^{-n/2-1}f\right)\\
&+r^{-2}\Delta_Nf\\
=&-\partial_r^2f-\frac{n}{2}\left(\frac{n}{2}+1\right)r^{-2}f+r^{-2}\frac{n^2}{2}f+r^{-2}\Delta_Nf\\
=&-\partial_r^2f+r^{-2}\left(\frac{n}{2}\left(\frac{n}{2}-1\right)f+\Delta_Nf\right).
\end{split}
\end{equation*}

Finally
\begin{equation}\label{Laplace on infinite cone}
T=-\partial_r^2+r^{-2}\left(\frac{n}{2}\left(\frac{n}{2}-1\right)+\Delta_N\right).
\end{equation}

Note that since
$$
\frac{n}{2}(\frac{n}{2}-1)+\Delta_N=\left(\frac{n}{2}-\frac12\right)^2+\Delta_N-\frac14\geq-\frac14,
$$ 
the operator $T$ is bounded below. It is also symmetric. The Friedrichs extension of $T$ satisfies the scaling property, see Lemma~\ref{integrability condition}. Below we deal only with this extension so to simplify the notation denote it by $T$.

\end{subsection}

\end{section}

\begin{section}{The expansions}

\begin{subsection}{The resolvent trace expansion}

Now we will consider resolvent of operator $T$. Since the manifold $(N,g_N)$ is compact, the Laplace-Beltrami operator $\Delta_N$ has discrete spectrum.  There is a basis of the Hilbert space $L^2(N)$ that consists of the corresponding eigenfunctions. Let $\lambda\in \Spec\Delta_N$ be eigenvalues of $\Delta_N$. Then 
$$
(T+z^2)^{-1}=\otimes_{\lambda}\left(-\partial_r^2+r^{-2}\left(\lambda+\frac{n}{2}\left(\frac{n}{2}-1\right)\right)+z^2\right)^{-1}\otimes_{\lambda}\pi_{\lambda},
$$
where $\pi_{\lambda}$
is the projection on the $\lambda$-eigenspace of $\Delta_N$.

The proof of the next lemma may appear similar to \cite[Lemma 4.1]{BS3}, but here our proof necessitates non-trivial subtle differences.
\begin{lemm}\label{Bessel functions}
Let $\nu=\sqrt{\lambda+(n-1)^2/4}$, $\text{Im }z^2\neq0$ and $0<r_1\leq r_2<\infty$, then the resolvent 
$$
\left(-\partial_r^2+r^{-2}\left(\lambda+\frac{n}{2}\left(\frac{n}{2}-1\right)\right)+z^2\right)^{-1}
$$
is an integral operator with kernel given by
$$
\left(-\partial_r^2+r^{-2}\left(\lambda+\frac{n}{2}\left(\frac{n}{2}-1\right)\right)+z^2\right)^{-1}(r_1,r_2)=(r_1r_2)^{1/2}I_\nu(r_1z)K_\nu(r_2z),
$$
where $I_\nu(r_1z)$ and $K_\nu(r_2z)$ are the modified Bessel functions of the first and second type respectively.
\end{lemm}

\begin{proof}
Let $v_1(r,z)$ and $v_2(r,z)$ be two linearly independent non-zero solutions of
\begin{equation}\label{equ}
\left(-\partial_r^2+r^{-2}\left(\lambda+\frac{n}{2}\left(\frac{n}{2}-1\right)\right)+z^2\right)u(r)=0,
\end{equation}
then by \cite[Theorem XIII.3.16]{DS}, the resolvent is an integral operator and the kernel of the resolvent is
\begin{equation}\label{formula_from_DS}
\begin{split}
\left(-\partial_r^2+r^{-2}\left(\lambda+\frac{n}{2}\left(\frac{n}{2}-1\right)\right)+z^2\right)^{-1}(r_1,r_2)\\
=(v_1'v_2-v_1v_2')^{-1}(r_1,z)v_1(r_1,z)v_2(r_2,z),
\end{split}
\end{equation}
for $0<r_1<r_2<\infty$.

We now find $v_1$ and $v_2$ by solving (\ref{equ}). Put 
$$
u(r)=:r^{1/2}w(r),
$$
then
$$
\partial_ru(r)=\frac12r^{-1/2}w(r)+r^{1/2}\partial_rw(r)
$$
and
$$
\partial_r^2u(r)=-\frac1 4 r^{-3/2}w(r)+r^{-1/2}\partial_r w(r)+r^{1/2}\partial_r^2w(r).
$$
Hence (\ref{equ}) becomes the modified Bessel equation
$$
\left(r^2\partial_r^2+r\partial_r-\left(r^2z^2+\lambda+\frac{(n-1)^2}{4}\right)\right)w(r)=0.
$$
Since $\nu=\sqrt{\lambda+(n-1)^2/4}$, the general solution is generated by the modified Bessel functions
$$
w(r)=C_1I_{\nu}(rz)+C_2K_{\nu}(rz),
$$
where $C_1,C_2\in\mathbb{R}$. Substitute $w(r)=r^{-1/2}u(r)$ to obtain
$$
u(r)=C_1r^{1/2}I_{\nu}(rz)+C_2r^{1/2}K_{\nu}(rz).
$$

Note that the modified Bessel function of the first kind, $I_{\nu}(r)$ with $\nu>0$, grows exponentially as $r\to\infty$. It tends to zero as $r\to0$. On the other hand the modified Bessel function of the second kind $K_{\nu}(r)$ tends to zero as $r\to\infty$, and grows as $r\to0$. Using the boundary condition at $r=0$ and the boundary condition at infinity $\lim_{r\to\infty} u(r)=0$, we derive two linearly independent solutions of (\ref{equ})
$$
v_1(r)=r^{1/2}I_{\nu}(rz)
$$
and
$$
v_2(r)=r^{1/2}K_{\nu}(rz).
$$

To find the kernel of the resolvent we use the formula (\ref{formula_from_DS}). Since $I_{\nu}'(x)K_{\nu}(x)-I_{\nu}(x)K_{\nu}'(x)=x^{-1}$,
\begin{equation}
\begin{split}
&(v_1'v_2-v_1v_2')^{-1}(r_1,z)\\
=&\left(\frac12r_1^{-1/2}I_{\nu}'(r_1z)+zr_1^{1/2}I'_{\nu}(r_1z)\right)r^{1/2}K_{\nu}(r_1z)\\
&-r_1^{1/2}I_{\nu}(r_1z)\left(\frac12r_1^{-1/2}K_{\nu}'(r_1z)+zr_1^{1/2}K'_{\nu}(r_1z)\right)\\
=&r_1z\big(I_{\nu}'(r_1z)K_{\nu}(r_1z)-I_{\nu}(r_1z)K_{\nu}'(r_1z)\big)
=1.
\end{split}
\end{equation}
By (\ref{formula_from_DS})

\begin{equation}
\begin{split}
&\left(-\partial_r^2+r^{-2}\left(\lambda+\frac{n}{2}\left(\frac{n}{2}-1\right)\right)+z^2\right)^{-1}(r_1,r_2)\\
=&(r_1r_2)^{1/2}I_\nu(r_1z)K_\nu(r_2z).
\end{split}
\end{equation}
\end{proof}

Since $(M,g)$ is a non-complete manifold, $T^{-1}$ may be not trace class, but for some $\varphi_1,\varphi_2\in C^{\infty}_c(M)$
$$
\varphi_1 T^{-1}\varphi_2\in C_p(L^2((0,\varepsilon),L^2(N))),
$$
where $C_p(H)$ is the $p$-th Schatten class of operators. It follows from the resolvent identity \cite[Theorem~5.13]{W} and the H\"older inequality for Schatten norms, that for $d>p$ and $\varphi\in C^{\infty}_c(M)$
$$
\varphi(T+z^2)^{-d}\in C_1(L^2(M)),
$$
with uniform trace norm estimate in
$$
\{z\in\mathbb{C} \;|\; |argz|<\delta\}, \;\;\;\;\;\;\;\;\;\; 0<\delta<\pi/2.
$$

It is shown in \cite[pp. 400-409]{BS2} that for any function $\varphi\in C^{\infty}_c(\mathbb{R})$ the operator $\varphi(r)(T+z^2)^{-d}$ is trace class for $d>\dim M/2=m/2$.

Let $\varphi$ be a smooth function on $(M,g)$ with support in the neighbourhood $(U,g_{\text{conic}})$ of the singularity, such that it depends only on the radial coordinate $r$. For each fixed $r\in(0,\varepsilon)$ and $p=(r,x^1,\dots,x^n)\in M$ by (\ref{local_resolvent_expansion}), we have the expansion of the heat kernel along the diagonal

\begin{equation}\label{local resolvent exp}
\begin{split}
&\varphi(r)(T+z^2)^{-d}(p)\sim_{|z|\to\infty}\\
&(4\pi)^{-\frac{m}{2}}\varphi(r)\sum_{j=0}^{\infty}z^{-2d+m-2j}u_{j}(p)\frac{\Gamma(-\frac{m}{2}+d+j)}{(d-1)!}.
\end{split}
\end{equation}

\begin{prop} \label{interior}
Let $\varphi(r)$ be a smooth function with compact support near $r=0$ and with $\varphi(r)\equiv1$ in a small neighbourhood of $r=0$. Let $d>m/2$ and $z\in\mathbb{C}$ such that $|arg(z)|<\frac{\pi}{2}$, then
\begin{equation*}
\begin{split}
&\tr (\varphi(r)(T+z^2)^{-d})\\
&=\int_0^{\infty}tr_{L^2(N)}(\varphi(r)(T+z^2)^{-d})dr\\
&=\int_0^{\infty}\varphi(r)\frac{r^{2d-1}}{(d-1)!}\left(-\frac{1}{2rz}\frac{\partial}{\partial (rz)}\right)^{d-1}\sum_{\nu}I_\nu(rz)K_\nu(rz)dr\\
&=\int_0^{\infty}\varphi(r)\sigma(r,rz)dr.
\end{split}
\end{equation*}
Above we sum over all $\nu=\sqrt{\lambda+(n-1)^2/4}$ such that $\lambda\in\Spec\Delta_N$, and 
$$
\sigma(r,rz):=\frac{r^{2d-1}}{(d-1)!}\left(-\frac{1}{2rz}\frac{\partial}{\partial (rz)}\right)^{d-1}\sum_{\nu}I_\nu(rz)K_\nu(rz).
$$
\end{prop}

\begin{proof}
The first equality follows from Lemma \ref{trace_lemma}. To prove that second equality we use Lemma \ref{Bessel functions} and the formula
\begin{align*}
(T+z^2)^{-d}
=&\frac{1}{(d-1)!}\left(-\frac{1}{2z}\frac{\partial}{\partial z}\right)^{d-1}(T+z^2)^{-1}.
\end{align*}
We obtain
\begin{align*}
tr_{L^2(N)}(T+z^2)^{-d}
=&\frac{1}{(d-1)!}\left(-\frac{1}{2z}\frac{\partial}{\partial z}\right)^{d-1}tr_{L^2(N)}(T+z^2)^{-1}\\
=&\frac{r^{2d-1}}{(d-1)!}\left(-\frac{1}{2rz}\frac{\partial}{\partial(rz)}\right)^{d-1}\sum_{\nu}I_\nu(rz)K_\nu(rz).
\end{align*}
\end{proof}

In Proposition \ref{interior}, we define the following function

\begin{equation}
\sigma(r,\zeta)
=tr_{L^2(N)}(T+\zeta^2/r^2)^{-d}.
\end{equation}

By (\ref{local resolvent exp}), we have an asymptotic expansion
\begin{align*}
&tr_{L^2(N)}(T+\zeta^2/r^2)^{-d}\sim_{\zeta\to\infty}\\
&(4\pi)^{-\frac{m}{2}}\sum_{j=0}^{\infty}
(\zeta/r)^{-2d+m-2j}\int_Nr^{m-1}u_{j}(p)\dvol_N\frac{\Gamma(-\frac{m}{2}+d+j)}{(d-1)!},
\end{align*}
hence

\begin{equation}\label{sigma_expansion}
\sigma(r,\zeta)
\sim_{\zeta\to\infty}\sum_{j=0}^{\infty}\zeta^{-2d+m-2j}\sigma_{j}(r),
\end{equation}
where

\begin{equation}\label{sigma expansion terms}
\begin{split}
\sigma_{j}(r)
=&(4\pi)^{-\frac{m}{2}}r^{2d-m+2j}\int_Nr^{m-1}u_{j}(p)\dvol_N\frac{\Gamma(-\frac{m}{2}+d+j)}{(d-1)!}\\
=&(4\pi)^{-\frac{m}{2}}r^{2d-1+2j}\int_Nu_{j}(p)\dvol_N\frac{\Gamma(-\frac{m}{2}+d+j)}{(d-1)!}.
\end{split}
\end{equation}
In particular we compute $\sigma_0(r)$ and $\sigma_1(r)$. By  Theorem \ref{local_heat_expansion}, $u_0(p)\equiv1$ and $u_1(p)=\frac16\tilde{\Scal}(p)$, where $\tilde{\Scal}(p)$ is the scalar curvature on $(M,g)$, therefore

$$
\sigma_{0}(r)=(4\pi)^{-\frac{m}{2}}r^{2d-1}\frac{\Gamma(d-\frac{m}{2})}{(d-1)!}\Vol(N)
$$
and

\begin{equation*}
\begin{split}
\sigma_{1}(r)
=&(4\pi)^{-\frac{m}{2}}\frac{\Gamma(d-\frac{m}{2}+1)}{(d-1)!}r^{2d-1+2}\int_Nu_1(p)\dvol_N\\
=&(4\pi)^{-\frac{m}{2}}\frac{\Gamma(d-\frac{m}{2}+1)}{(d-1)!}r^{2d+1}\int_N\frac{\tilde{\Scal}(p)}{6}\dvol_N\\
=&(4\pi)^{-\frac{m}{2}}\frac{\Gamma(d-\frac{m}{2}+1)}{6(d-1)!}r^{2d-1}\int_N(\Scal(x)-n(n-1))\dvol_N,
\end{split}
\end{equation*}
where as before $\tilde{\Scal}(p)$ is the scalar curvature of $(M,g)$ at $p\in M$ and $\Scal(x)$ is the scalar curvature of $(N,g_N)$ at $x\in N$.


Now we show that the function $\sigma(r,\zeta)$ satisfies the three conditions of the Singular Asymptotics Lemma (Lemma \ref{SAL}).  First we observe that $\sigma(r,\zeta)$ is $C^{\infty}$ with respect to $r$. Moreover according to \cite[Section 3]{BS2}, $\sigma(r,\zeta)$ has analytic derivatives with respect to $\zeta$. The second condition of Lemma \ref{SAL} is satisfied due to $(\ref{sigma_expansion})-(\ref{sigma expansion terms})$. The following lemma gives the proof of the third property. Denote $\partial^j\sigma(r,\zeta):=\frac{\partial^j}{\partial r^j}\sigma(r,\zeta)$.

\begin{lemm}(Integrability condition)\label{integrability condition}
For $j\in\mathbb{N}$ and $0<|arg\zeta|<\delta<\pi/2$ with $|\zeta|=c_0$, there is a constant $c(c_0,j)$ such that the following is satisfied uniformly for $0\leq t\leq1$
$$
\int_0^1\int_0^1s^j\left|\partial^j\sigma(st,s\zeta)\right|dsdt\leq c(c_0,j).
$$
\end{lemm}

\begin{proof}

Let $t\in[0,1]$. Define unitary scaling operator
$$
(U_tf)(r):=t^{\frac12}f(tr), \;\;\;\;\; f\in L^2(0,\infty).
$$
Then
\begin{align*}
U_tr^l=t^lr^lU_t, && U_t\partial_r=t^{-1}\partial_rU_t,
\end{align*}
where $l\in\mathbb{N}$. By (\ref{Laplace on infinite cone}),
$$
U_tT=t^{-2}TU_t,
$$
note that this is true, because $T$ is the Friedrichs extension, \cite[Section~2]{L}.
Define an operator in $L^2((0,\varepsilon),L^2(N))$
$$
T_t:=t^2U_tTU_t^\dag.
$$
By Lemma~\ref{trace_lemma}, $(T_t+z^2)^{-d}$ has a continuous kernel. Denote
\begin{align}\label{tmp sigma}
\sigma_t(r,rz):=\tr_{L^2(N)}(T_t+z^2)^{-d},
\end{align}
in particular $\sigma_1(r,rz)=\sigma(r,rz)$.

We have $U_t(T+z^2)^{-1}=t^2(T_t+(tz)^2)^{-1}U_t$, therefore
\begin{align}\label{operator scaling property}
U_t(T+z^2)^{-d}=t^{2d}(T_t+(tz)^2)^{-d}U_t.
\end{align}
By (\ref{tmp sigma}) and (\ref{operator scaling property}), we obtain the following scaling property 
\begin{align}
\sigma(rt,\zeta)=t^{2d-1}\sigma_t(r,\zeta).
\end{align}
Hence $\sigma(st,s\zeta)=t^{2d-1}\sigma_t(s,\zeta)$. By the chain rule
$$
\frac{\partial}{\partial(st)}\sigma(st,s\zeta)
=\frac{\partial t}{\partial(st)}\frac{\partial}{\partial t}\sigma(st,s\zeta)
=s^{-1}\frac{\partial}{\partial t}\sigma(st,s\zeta),
$$
therefore
$$
\frac{\partial^j}{\partial (st)^j}\sigma(st,s\zeta)
=s^{-j}\frac{\partial^j}{\partial t^j}\sigma(st,s\zeta).
$$
Then
$$
s^j\partial^j\sigma(st,s\zeta)
=\partial^j_t\sigma(st,s\zeta)
=\partial^j_t(t^{2d-1}\sigma_t(s,\zeta)).
$$

Now choose a function $\varphi\in C^{\infty}_c(\mathbb{R})$ such that $\varphi\geq0$ and $\varphi\equiv1$ in the interval $[0,1]$. Then

\begin{align*}
\int_0^1s^j|\sigma^{(j)}(st,s\zeta)|ds
&=\int_0^1|\partial_t^j(t^{2d-1}\sigma_t(s,\zeta))|ds\\
&\leq||\partial_t^j(t^{2d-1}\varphi(T+z^2)^{-d})||_{tr}
\leq C_0,
\end{align*}
where the last inequalities follow from Lemma \ref{trace_lemma}. 
\end{proof}


Now we can apply the Singular Asymptotics Lemma (Lemma \ref{SAL}) to obtain

\begin{prop}\label{applySAL}
\begin{equation}\label{theexpansion}
\begin{split}
\int_0^{\infty}\varphi(r)\sigma(r,\zeta)dr
\sim&\sum_{l=0}^{\infty}z^{-l-1}\frac{1}{l!}\fint_0^{\infty}\zeta^l\partial^l_r\bigg(\sigma(r,\zeta)\varphi(r)\bigg)|_{r=0}d\zeta\\
+&\sum_{j=0}^{\infty}\fint_0^{\infty}\sigma_{j}(r)(rz)^{-2d+m-2j}\varphi(r)dr\\
+&\sum_{l=\frac{m}{2}-d+1}^{\infty}z^{-2d+m-2l}\log z\frac{\partial^{2d-m+2l-1}_r\bigg(\sigma_{l}(r)\varphi(r)\bigg)|_{r=0}}{(2d-m+2l-1)!}.
\end{split}
\end{equation}
\end{prop}

\begin{rema}
If $M$ is an odd-dimensional manifold, the last sum in (\ref{theexpansion}) is zero, because $\sigma_l$ makes sense only for $l\in\mathbb{N}_0$.
\end{rema}


Since $\varphi(r)\equiv1$ near $r=0$, all its derivatives vanish at zero. Therefore the first sum simplifies and we use Proposition \ref{interior} to obtain

\begin{equation}\label{first summand}
\begin{split}
&\sum_{l=0}^{\infty}z^{-l-1}\frac{1}{l!}\fint_0^{\infty}\zeta^l\partial^l_r\sigma(r,\zeta)|_{r=0}d\zeta\\
=&z^{-2d}\fint_0^{\infty}\zeta^{2d-1}\left(-\frac{1}{2rz}\frac{\partial}{\partial (rz)}\right)^{d-1}\sum_{\nu}I_\nu(rz)K_\nu(rz)d\zeta.
\end{split}
\end{equation}
Note that this term in the resolvent trace expansion gives a contribution to the constant term in the heat trace expansion. We simplify the terms in (\ref{theexpansion}) in the next section.

\end{subsection}

\begin{subsection}{The heat trace expansion}\label{heat trace expansion}


Let $f$ be a meromorphic function with the Laurent series expansion at a point $z_0$
$$
f(z)=\sum_{k=-k_0}^\infty\Res_{-k}f(z_0)(z-z_0)^{k},
$$
where $\Res_{-k}f(z_0)=\frac{1}{2\pi i}\int_{\gamma}\frac{f(z)dz}{(z-z_0)^{k+1}}$, for a positive oriented path $\gamma$ enclosing $z_0$ and lying in an annulus, in which $f(z)$ is holomorphic. By these notations, $\Res_1$ is the residue of the function and $\Res_0$ is the regular analytic continuation. Let $f$ be analytic at $z_0$, and assume that $g$ has a simple pole at $z_0$. Then
\begin{align}\label{residue formula}
\Res_0(fg)(z_0)=f(z_0)\Res_0g(z_0)+f'(z_0)\Res_1g(z_0).
\end{align}

We analyse the first summand in Proposition \ref{applySAL} given by (\ref{first summand}). Denote by $p:=2d-1$. Then by the Mellin transform \cite[p.123]{O}, we obtain
\begin{equation*}
\begin{split}
\fint_0^{\infty}\zeta^{p}\left(-\frac{1}{2\zeta}\frac{\partial}{\partial\zeta}\right)^{d-1}I_\nu(\zeta)K_\nu(\zeta)d\zeta\\
=\frac{1}{4\sqrt{\pi}}\frac{\Gamma(\nu-d+\frac{p+3}{2})\Gamma(d-1-\frac{p}{2})\Gamma(\frac{p+1}{2})}{\Gamma(d+1+\nu-\frac{p+3}{2})}.
\end{split}
\end{equation*}
Set $\frac{p}{2}=l+d-\frac32$ and denote by $\Res_0f(l)|_{l=1}$ the regular analytic continuation of $f(l)$ at $l=1$, i.e. the constant term in the Laurent expansion. Following Definition~\ref{definition of regularized integral},
\begin{align*}
&\fint_0^{\infty}\zeta^{p}\left(-\frac{1}{2\zeta}\frac{\partial}{\partial\zeta}\right)^{d-1}I_\nu(\zeta)K_\nu(\zeta)d\zeta\\
&=\Res_0\left(\frac{1}{4\sqrt{\pi}}\Gamma(-\frac{l}{2})\Gamma(d+\frac{l}{2}-\frac{1}{2})\frac{\Gamma(\nu+\frac{l}{2}+\frac1 2)}{\Gamma(\nu-\frac{l}{2}+\frac{1}{2})}\right)|_{l=1}.
\end{align*}

The ratio of two Gamma functions is given in the proposition below. Let $B_j$ be the Bernoulli numbers and $C_j^i$ be the binomial coefficients
$$
B_0=1, B_j=-\sum_{i=0}^{j-1}C_j^i\frac{B_i}{j-i+1}, \;\;\; j\geq1.
$$
In particular $B_1=-1/2$, $B_2=1/6$, $B_4=-1/30$, $B_6=1/42$ and $B_{2j+1}=0$ for $j\geq1$.

\begin{prop}\label{ratioofGamma}
We have that
$$
\frac{\Gamma(\nu-s+1)}{\Gamma(\nu+s)}
\sim_{\nu\to\infty}
\nu^{1-2s}\bigg(1+s\sum_{j\geq1}j^{-1}B_{2j}\nu^{-2j}\bigg)+O(s^2).
$$
\end{prop}
\begin{proof}
According to \cite[Chapter XII p.251]{WW} \footnote{Note that in this book an old notation of the Bernoulli numbers is used $B^{old}_j$, in particular $B^{old}_1=1/6$, $B^{old}_2=1/30$, $B^{old}_3=1/42$; and here we use the modern notation $B^{new}_j:=B_j$, they satisfy the relation $B^{old}_j=(-1)^{j-1}B^{new}_{2j}$  and $B^{new}_{2j+1}=0$ for $j\geq1$; that is why our formula is slightly different from the one in the book.},
$$
\log\Gamma(z)\sim_{z\to\infty}
(z-\frac1 2)\log z-z+\frac1 2\log(2\pi)
+\sum_{j\geq1}\frac{1}{2j(2j-1)}\frac{B_{2j}}{z^{2j-1}}.
$$
Consequently,
\begin{equation*}
\begin{split}
\log\Gamma(\nu+s)\sim_{\nu\to\infty}
&(\nu+s-\frac1 2)\log(\nu+s)-(\nu+s)+\frac1 2\log(2\pi)\\
&+\sum_{j\geq1}\frac{1}{2j(2j-1)}\frac{B_{2j}}{(\nu+s)^{2j-1}}\\
=&(\nu+s-\frac1 2)(\log\nu+\log(1+s/\nu))-(\nu+s)+\frac1 2\log(2\pi)\\
&+\sum_{j\geq1}\frac{1}{2j(2j-1)}\frac{B_{2j}}{(\nu+s)^{2j-1}},
\end{split}
\end{equation*}
where for the last equality we use $\log(\nu+s)=\log\nu+\log(1+s/\nu)$.
Analogously, using $\log(\nu-s)=\log\nu+\log(1-s/\nu)$, we obtain
\begin{equation*}
\begin{split}
\log\Gamma(\nu-s)\sim_{\nu\to\infty}
&(\nu-s-\frac1 2)(\log\nu+\log(1-s/\nu))-(\nu-s)+\frac1 2\log(2\pi)\\
&+\sum_{j\geq1}\frac{1}{2j(2j-1)}\frac{B_{2j}}{(\nu-s)^{2j-1}}.
\end{split}
\end{equation*}
Furthermore,
\begin{equation*}
\begin{split}
&\log\frac{\Gamma(\nu-s)}{\Gamma(\nu+s)}
=\log\Gamma(\nu-s)-\log\Gamma(\nu+s)
\sim_{\nu\to\infty}\\
&-2s\log\nu+(\nu-s-\frac1 2)\log(1-\frac{s}{\nu})
-(\nu+s-\frac1 2)\log(1+\frac{s}{\nu})+2s\\
&+\sum_{j\geq1}\frac{1}{2j(2j-1)}B_{2j}
\left(\frac{1}{(\nu-s)^{2j-1}}-\frac{1}{(\nu+s)^{2j-1}}\right)\\
\sim&-2s\log\nu+(\nu-s-\frac1 2)(-\frac{s}{\nu})
-(\nu+s-\frac1 2)(\frac{s}{\nu})+2s\\
&+\sum_{j\geq1}\frac{1}{2j(2j-1)}B_{2j}
\left(\frac{(\nu+s)^{2j-1}-(\nu-s)^{2j-1}}{(\nu^2-s^2)^{2j-1}}\right)+O(s^2)\\
=&-2s\log\nu+\frac{s}{\nu}+s\sum_{j\geq1}j^{-1}B_{2j}\nu^{-2j}+O(s^2).
\end{split}
\end{equation*}

Finally we obtain
\begin{equation*}
\begin{split}
&\frac{\Gamma(\nu-s+1)}{\Gamma(\nu+s)}
=(\nu-s)\frac{\Gamma(\nu-s)}{\Gamma(\nu+s)}\\
&\sim(\nu-s)\nu^{-2s}\left(1+\frac{s}{\nu}+s\sum_{j\geq1}j^{-1}B_{2j}\nu^{-2j}\right)+O(s^2)\\
&\sim\nu^{1-2s}\left(1+s\sum_{j\geq1}j^{-1}B_{2j}\nu^{-2j}\right)+O(s^2).
\end{split}
\end{equation*}

\end{proof}

Using Proposition \ref{ratioofGamma} we compute
\begin{equation}\label{b_1}
\begin{split}
b_1
:=&\Res_0\left(\sum_{\lambda\in\Spec\Delta_N}\frac{\Gamma(-\frac{l}{2})\Gamma(d+\frac{l}{2}-\frac{1}{2})}{4\sqrt{\pi}(d-1)!}\frac{\Gamma(\sqrt{\lambda+(n-1)^2/4}+\frac{l}{2}+\frac1 2)}{\Gamma(\sqrt{\lambda+(n-1)^2/4}-\frac{l}{2}+\frac{1}{2})}\right)|_{l=1}\\
=&\Res_0\Bigg(\frac{\Gamma(-\frac{l}{2})\Gamma(d+\frac{l}{2}-\frac{1}{2})}{4\sqrt{\pi}\Gamma(d)}
\sum_{\lambda\in\Spec\Delta_N}\bigg(\left(\lambda+\frac{(n-1)^2}{4}\right)^{\frac{l}{2}}\\
&-\sum_{j\geq1}j^{-1}B_{2j}\left(\frac{l}{2}-\frac1 2\right)\left(\lambda+\frac{(n-1)^2}{4}\right)^{-j+\frac{l}{2}}\bigg)\Bigg)|_{l=1},
\end{split}
\end{equation}

To find the regular analytic continuation at $l=1$ of the function above we set some notations.

\begin{defi}
Let $h\in\mathbb{R}$. The {\itshape shifted by $h$ spectral zeta function} of $(N,g_N)$ is, by definition,
$$
\zeta^h_N(s):=\sum_{\lambda\in\Spec\Delta_N}\left(\lambda+h^2\right)^{-s}.
$$
\end{defi}

To compute (\ref{b_1}), we need to find the residues of the shifted zeta function. Denote by $a^N_j$, $j\in\mathbb{N}_0$ the heat trace expansion coefficients on the closed manifold $(N,g_N)$
\begin{align}\label{expansion on N}
\tr e^{-t\Delta_N}\sim_{t\to0+}(4\pi t)^{-n/2}\sum_{j=0}^{\infty}a^N_jt^j.
\end{align}

\begin{lemm}\label{poles of shifted zeta}
If $n$ is odd, then $\zeta^h_N(s)$ is a meromorphic function with simple poles at $s=\frac{n}{2}-l$, $l\in\mathbb{N}_0$ with residue 
$$
\Res_1\zeta^h_N\left(\frac{n}{2}-l\right)
=\frac{1}{(4\pi)^{n/2}\Gamma(\frac{n}{2}-l)}\sum_{i=0}^l(-1)^i\frac{h^{2i}}{i!}a^N_{l-i}.
$$
If $n$ is even, then the same holds, but there are no poles for $l=\frac{n}{2}+j, j\in\mathbb{N}_0$.
\end{lemm}

\begin{proof}

Using the Mellin transform we obtain
\begin{align*}
&\zeta^h_N(s)
=\sum_{\lambda\in Spec\Delta_N}\left(\lambda+h^2\right)^{-s}
=\frac{1}{\Gamma(s)}\sum_{\lambda\in Spec\Delta_N}\int_0^\infty t^{s-1}e^{-t(\lambda+h^2)}dt\\
=&\frac{1}{\Gamma(s)}\sum_{\lambda\in Spec\Delta_N}\int_0^1 t^{s-1}e^{-t\lambda}e^{-th^2}dt
+\frac{1}{\Gamma(s)}\sum_{\lambda\in Spec\Delta_N}\int_1^\infty t^{s-1}e^{-t(\lambda+h^2)}dt.
\end{align*}
The latter integral is an entire function in $s$. To compute the first integral we use the Taylor expansion
$$
e^{-th^2}\sim_{t\to0}\sum_{i=0}^\infty(-1)^i\frac{h^{2i}}{i!}t^i,
$$
and the heat trace expansion (\ref{expansion on N})). Since Gamma function is nowhere zero, we obtain
\begin{align*}
Res_1\zeta^h_N(s_0)
=&Res_1\left(\frac{1}{(4\pi)^{n/2}\Gamma(s_0)}\sum_{i=0}^\infty(-1)^i\frac{h^{2i}}{i!}\sum_{j=0}^{\infty}\frac{a^N_j}{(s_0+i+j-\frac{n}{2})}\right).
\end{align*}
If $n$ is odd, poles are at $s_0=\frac{n}{2}-i-j$, for $i,j\in\mathbb{N}_0$. If $n$ is even, poles are at $s_0=\frac{n}{2}-i-j$, for $i+j<\frac{n}{2}$, because the Gamma function has simple poles at non-positive integers. Set $l:=i+j$, then

\begin{align*}
Res_1\zeta^h_N\left(\frac{n}{2}-l\right)
=&\frac{1}{(4\pi)^{n/2}\Gamma(\frac{n}{2}-l)}\sum_{i=0}^l(-1)^i\frac{h^{2i}}{i!}a^N_{l-i}.
\end{align*}
\end{proof}

Note that this agrees with

\begin{prop}\cite[p.3]{V}\label{poles on compact}
The spectral zeta function on $(N,g_N)$, $\zeta_N(s)$, is a meromorphic function with simple poles at $s=\frac{n}{2}-l$ for $l\in\mathbb{N}_0$ with the residue
$$
Res_1\zeta_N(s)=\frac{a^N_{\frac{n}{2}-s}}{(4\pi)^{n/2}\Gamma(s)}.
$$
\end{prop}

We now compute (\ref{b_1}). Set
$$
f(s):=\frac{\Gamma(\frac12+s)\Gamma(d-s-1)}{4\sqrt{\pi}\Gamma(d)}.
$$
Denote also
$$
g(s):=-\sum_{j\geq1}j^{-1}B_{2j}\zeta^{\frac{n-1}{2}}_N(j+s/2).
$$
Since $f(-1)=-\frac1 2$ and $\Gamma(-\frac1 2)=-2\sqrt{\pi}$ and $f'(-1)=\frac{\Gamma'(-\frac1 2)}{4\sqrt{\pi}}+\frac{\Gamma'(d)}{2\Gamma(d)}$, by (\ref{b_1}) and (\ref{residue formula}), we have
\begin{equation*}
\begin{split}
b_1
=&\Res_0\left(f(s)\zeta^{\frac{n-1}{2}}_N(s/2)+f(s)\left(-\frac{s}{2}-\frac1 2\right)g(s)\right)|_{s=-1}\\
=&-\frac{1}{2}\Res_0\zeta^{\frac{n-1}{2}}_N(-1/2)\\
&+\left(\frac{\Gamma'(-\frac1 2)}{4\sqrt{\pi}}+\frac{\Gamma'(d)}{2\Gamma(d)}\right)\Res_1\zeta^{\frac{n-1}{2}}_N(-1/2)
+\frac1 4\Res_1g(-1).
\end{split}
\end{equation*}
By Lemma~\ref{poles of shifted zeta},

\begin{equation}\label{b2m}
\begin{split}
b_1
=&-\frac{1}{2}\Res_0\zeta^{\frac{n-1}{2}}_N(-1/2)
+\left(\frac{\Gamma'(-\frac1 2)}{4\sqrt{\pi}}+\frac{\Gamma'(d)}{2\Gamma(d)}\right)\Res_1\zeta^{\frac{n-1}{2}}_N(-1/2)\\
&-\frac1 4\sum_{1\leq j\leq\frac{n+1}{2}}j^{-1}B_{2j}\Res_1\zeta^{\frac{n-1}{2}}_N(j-1/2).
\end{split}
\end{equation}
If $n$ is even, the above sum is to be interpreted as sum over all integers $1\leq j\leq\frac{n+1}{2}$.


Now consider the second sum in the expansion (\ref{theexpansion}). By (\ref{sigma expansion terms}),

\begin{equation}\label{tmp summation}
\begin{split}
&\sum_{j=0}^{\infty}
\fint_0^{\infty}\sigma_{j}(r)(rz)^{-2d+m-2j}\varphi(r)dr\\
=&(4\pi)^{-\frac{m}{2}}
\sum_{j=0}^{\infty}
\frac{\Gamma(-\frac{m}{2}+d+j)}{(d-1)!}z^{-2d+m-2j}\fint_0^{\infty}r^{m-1}\varphi(r)\int_N u_j(p)\dvol_Ndr.
\end{split}
\end{equation}
By \cite[Chapter III, Lemma E.IV.5]{BGM}, $u_j(r,x)=r^{-2j}\hat{u}_j(r,x)$, where $\hat{u}_j(r,x)$ is smooth with respect to $r$. Hence for $m-1-2j\geq0$, the integrals in the above sum need no regularisation, i.e. for $j\leq m/2-1$ we have 
$$
\fint_0^{\infty}r^{m-1}\varphi(r)\int_N u_j(p)\dvol_Ndr=\int_0^{\infty}r^{m-1}\varphi(r)\int_N u_j(p)\dvol_Ndr.
$$
The sum (\ref{tmp summation}) gives the following terms in the heat trace expansion
\begin{equation}\label{proof of claim 1}
\begin{split}
(4\pi)^{-\frac{m}{2}}
&\sum_{j=0}^{\infty}t^{-\frac{m}{2}+j}\fint_0^{\infty}r^{m-1}\varphi(r)\int_N u_j(p)\dvol_Ndr\\
=(4\pi)^{-\frac{m}{2}}
&\sum_{j=0}^{\infty}t^{-\frac{m}{2}+j}\fint_Mu_{j}(p)\varphi(r)\dvol_M,
\end{split}
\end{equation}
where summands with $j\leq m/2-1$ need no regularisation.

We sum this with the expansion away from the singularity to obtain
\begin{equation*}
\begin{split}
(4\pi)^{-\frac{m}{2}}
&\sum_{j=0}^{\infty}t^{-\frac{m}{2}+j}\bigg(\fint_Mu_{j}(p)\varphi(r)\dvol_M
+\int_Mu_{j}(p)(1-\varphi(r))\dvol_M\bigg)\\
&=\sum_{j=0}^{\infty}t^{-\frac{m}{2}+j}\fint_Mu_{j}(p)\dvol_M.
\end{split}
\end{equation*}


Now we simplify the logarithmic terms in Proposition \ref{applySAL}

\begin{equation*}
\begin{split}
L:=&\sum_{j=\frac{m}{2}-d+1}^{\infty}
z^{-2d+m-2j}\log z\frac{\partial^{2d-m+2j-1}_r\bigg(\sigma_{j}(r)\varphi(r)\bigg)|_{r=0}}{(2d-m+2j-1)!}\\
=&\sum_{j=\frac{m}{2}-d+1}^{\infty}
z^{-2d+m-2j}\log z\frac{\Gamma(-\frac{m}{2}+d+j)}{(d-1)!(2d-m+2j-1)!}\times\\
&\times(4\pi)^{-\frac{m}{2}}
\partial^{2d-m+2j-1}_r\bigg(r^{2d-1+2j}\int_N u_j(p)\dvol_N\bigg)|_{r=0}.
\end{split}
\end{equation*}
To continue the computation we note that for $j\geq m/2$ we have

\begin{equation*}
\begin{split}\partial^{2d-m+2j-1}_r\bigg(r^{2d-1+2j}\int_N u_j(p)\dvol_N\bigg)|_{r=0}\\
=\partial^{2j-m}_r\bigg(r^{2j}\int_N u_j(p)\dvol_N\bigg)|_{r=0},
\end{split}
\end{equation*}
therefore

\begin{equation*}
\begin{split}
L=&\sum_{j=\frac{m}{2}}^{\infty}
z^{-2d+m-2j}\log z\frac{\Gamma(-\frac{m}{2}+d+j)}{(d-1)!(2j-m)!}\times\\
&\times(4\pi)^{-\frac{m}{2}}
\partial^{2j-m}_r\bigg(r^{2j}\int_N u_j(p)\dvol_N\bigg)|_{r=0}\\
=&\sum_{l=0}^{\infty}
z^{-2d-2l}\log z\frac{\Gamma(d+l)}{(d-1)!(2l)!}(4\pi)^{-\frac{m}{2}}
\partial^{2l}_r\bigg(r^{m+2l}\int_N u_{\frac{m}{2}+l}(p)\dvol_N\bigg)|_{r=0}.
\end{split}
\end{equation*}

Hence the logarithmic part in the heat trace expansion coming from this term is

\begin{align*}
-(4\pi)^{-\frac{m}{2}}
\sum_{i=0}^{\infty}t^i
\log t\times
\frac{1}{2(2i)!}\partial^{2i}_r\bigg(r^{m+2i}\int_{N}u_{\frac{m}{2}+i}(p)\dvol_N\bigg)|_{r=0}.
\end{align*}

For $i>0$, the function $r^{m+2i}\int_{N}u_{\frac{m}{2}+i}(p)\dvol_N$ is smooth and has no terms of order $r^{2i}$. Hence the only nonzero logarithmic term may appear for $i=0$
\begin{equation}\label{proof of claim 3}
c:=-(4\pi)^{-\frac{m}{2}}\frac{1}{2}\bigg(r^{m}\int_Nu_{\frac{m}{2}}(p)\dvol_N\bigg)|_{r=0}\\
=\frac12\Res_1\zeta_N^{\frac{n-1}{2}}(-1/2).
\end{equation}

From the logarithmic term in the resolvent trace expansion we also get the following contribution to the constant term in the heat trace expansion

\begin{align}\label{b_2}
b_2
&=(4\pi)^{-\frac{m}{2}}\frac{1}{2}\frac{\Gamma'(d)}{\Gamma(d)}\bigg(r^{m}\int_Nu_{\frac{m}{2}}(p)dvol_N\bigg)|_{r=0}=-\frac12\frac{\Gamma'(d)}{\Gamma(d)}Res_1\zeta^{\frac{n-1}{2}}_N(-1/2).
\end{align}

Let $b:=b_1+b_2$. By (\ref{b2m}) and (\ref{b_2}),

\begin{equation}\label{constant term}
\begin{split}
b
=&-\frac{1}{2}\Res_0\zeta^{\frac{n-1}{2}}_N(-1/2)
+\frac{\Gamma'(-\frac1 2)}{4\sqrt{\pi}}\Res_1\zeta^{\frac{n-1}{2}}_N(-1/2)\\
&-\frac1 4\sum_{1\leq j\leq\frac{n+1}{2}}j^{-1}B_{2j}\Res_1\zeta^{\frac{n-1}{2}}_N(j-1/2).
\end{split}
\end{equation}

\end{subsection}

\begin{subsection}{Proof of Theorem \ref{main theorem}}
\label{proof of main theorem}

Now we are ready to prove Theorem \ref{main theorem}.

\begin{proof}[Proof of Theorem~\ref{main theorem}]

First, we show that if $m\geq4$, then the Laplace-Beltrami operator on $(M,g)$ is essentially self-adjoint. By (\ref{Laplace on infinite cone}),
$$
\Delta=-\partial_r^2+r^{-2}\left(\frac{n}{2}\left(\frac{n}{2}-1\right)+\Delta_N\right),
$$
where $n=m-1$ and $\Delta_N$ is the Laplace-Beltrami operator on $(N,g_N)$. By \cite[pp.703-704]{BS3}, $\Delta$ is essentially self-adjoint if
$$
\left(\frac{n}{2}\left(\frac{n}{2}-1\right)+\Delta_N\right)\geq\frac34,
$$
equivalently,
$$
n\geq3,
$$
otherwise $\Delta$ might have many self-adjoint extensions.

(a) By (\ref{proof of claim 1}), (\ref{proof of claim 3}) and (\ref{constant term}), we obtain the final formula

\begin{equation}
\tr e^{-t\Delta}\sim_{t\to0+}(4\pi t)^{-\frac{m}{2}}
\sum_{j=0}^{\infty}\tilde{a}_jt^j
+b
+c\log t.
\end{equation}
Above for $j\leq m/2-1$, we have $\tilde{a}_j=\int_Mu_j(p)\dvol_M$. For \\* $j>m/2-1$, we have the reguralized integrals $\tilde{a}_j=\fint_Mu_j(p)\dvol_M$.

(b) The constant term $b$ is given by (\ref{constant term}).

In the subsequent article we show that the constant term $b_{S^n}$ in the heat trace expansion on $(M_1,g_1)$ with the cross-section $(N,g_N)=(S^n,g_{\text{round}})$ is equal to zero. We also show that the constant term $b_{\mathbb{R}P^n}$ in the heat trace expansion on $(M_2,g_2)$ with the cross-section $(N,g_N)=(\mathbb{R}P^n,g_{\text{round}})$ is non-zero. Observe, that $S^n$ is the covering of $\mathbb{R}P^n$. By the first claim of this theorem, $\tilde{a}_j$ has an expression by the integral of the local data that is the curvature and its derivatives. Assume that the constant term $b$ can be written as an integral over local data. Then $b$ satisfies the multiplicative law for coverings, hence $b_{S^n}=2b_{\mathbb{R}P^n}$. It is a contradiction. We conclude that in general there is no expression of $b$ as an integral of local data.

(c) By (\ref{proof of claim 3}), we get
$$
c=\frac12\Res_1\zeta^{\frac{m-2}{2}}_N(-1/2).
$$
By Lemma \ref{poles of shifted zeta},
\begin{align*}
\Res_1\zeta^{\frac{m-2}{2}}_N(s)
=\Res_1\left(\frac{1}{(4\pi)^{\frac{m-1}{2}}\Gamma(s)}\sum_{k=0}^\infty(-1)^k\frac{(m-2)^{2k}}{2^{2k}k!}\sum_{j=0}^{\infty}\frac{a^N_j}{(s+k+j-\frac{m-1}{2})}\right).
\end{align*}

Setting $s=-\frac12$, we get the relation between the summation indices
$$
k+j=\frac{m}{2},
$$
hence 
\begin{align}\label{residue at -1/2}
\Res_1\zeta^{\frac{m-2}{2}}_N(-1/2)
=&\frac{1}{(4\pi)^{\frac{m-1}{2}}\Gamma(-\frac12)}\sum_{k=0}^{\frac{m}{2}}(-1)^k\frac{(m-2)^{2k}}{2^{2k}k!}a^N_{\frac{m}{2}-k},
\end{align}
and we simplify
\begin{align*}
\Res_1\zeta^{\frac{m-2}{2}}_N(-1/2)
=&-\frac{1}{(4\pi)^{\frac{m}{2}}}\sum_{k=0}^{\frac{m}{2}}(-1)^k\frac{(m-2)^{2k}}{2^{2k}k!}a^N_{\frac{m}{2}-k}\\
=&\frac{1}{(4\pi)^{\frac{m}{2}}}\sum_{k=0}^{\frac{m}{2}}(-1)^{k+1}\frac{(m-2)^{2k}}{2^{2k}k!}a^N_{\frac{m}{2}-k}.
\end{align*}
If $m$ is odd we have $a^N_{\frac{m}{2}-k}=0$ for $0\leq k<\frac{m}{2}$.

(d) Assume that $c=0$, then by (\ref{proof of claim 3}),
$$
\bigg(r^{m}\int_Nu_{\frac{m}{2}}(r,x)\dvol_N\bigg)|_{r=0}=0.
$$
Hence
\begin{align*}
\tilde{a}_{\frac{m}{2}}
=&\fint_Mu_{\frac{m}{2}}(p)\dvol_M
=\fint_{U}u_{\frac{m}{2}}(p)\dvol_M+\int_{M\setminus U}u_{\frac{m}{2}}(p)\dvol_M\\
=&\int_{M\setminus U}u_j(p)\dvol_M
=\int_Mu_j(p)\dvol_M.
\end{align*}
This finishes the proof.
\end{proof}

\end{subsection}

\end{section}

\end{document}